\documentclass[12pt]{elsarticle}
\usepackage[LGR,T1]{fontenc}
\usepackage[latin9]{inputenc}
\usepackage{color}
\usepackage{amsmath}
\usepackage{amsthm}
\usepackage{amssymb}
\usepackage[unicode=true,
 bookmarks=true,bookmarksnumbered=true,bookmarksopen=true,bookmarksopenlevel=3,
 breaklinks=false,pdfborder={0 0 1},backref=false,colorlinks=true]
 {hyperref}
\hypersetup{
 pdfauthor={Lei Yu},
 pdfborderstyle=,pdfborderstyle={},pdfborderstyle={},pdfborderstyle={},pdfborderstyle={},pdfborderstyle={},pdfborderstyle={},pdfborderstyle={},pdfborderstyle={},pdfborderstyle={},pdfborderstyle={},pdfborderstyle={},pdfborderstyle={},pdfborderstyle={},pdfborderstyle={},pdfborderstyle={},pdfborderstyle={},pdfborderstyle={},pdfborderstyle={},pdfborderstyle={},pdfpagelayout=OneColumn,pdfnewwindow=true,pdfstartview=XYZ,plainpages=false,linkcolor=blue,urlcolor=blue,citecolor=red,anchorcolor=blue,linkcolor=blue,urlcolor=blue,citecolor=red,anchorcolor=blue}

\makeatletter
\theoremstyle{plain}
\newtheorem{thm}{\protect\theoremname}
\theoremstyle{definition}
\newtheorem{defn}{\protect\definitionname}
\theoremstyle{plain}
\newtheorem{prop}{\protect\propositionname}
\theoremstyle{plain}
\newtheorem{lem}{\protect\lemmaname}
\theoremstyle{definition}
\newtheorem{problem}{\protect\problemname}
\theoremstyle{remark}
\newtheorem{claim}{\protect\claimname}

\usepackage{color}
\usepackage{fullpage}

\allowdisplaybreaks[1]
\providecommand{\definitionname}{Definition}
\providecommand{\lemmaname}{Lemma}
\providecommand{\problemname}{Problem}
\providecommand{\propositionname}{Proposition}
\providecommand{\theoremname}{Theorem}

\makeatother

\providecommand{\claimname}{Claim}
\providecommand{\definitionname}{Definition}
\providecommand{\lemmaname}{Lemma}
\providecommand{\problemname}{Problem}
\providecommand{\propositionname}{Proposition}
\providecommand{\theoremname}{Theorem}

\begin{document}

\title{Edge-Isoperimetric Inequalities and Ball-Noise Stability: Linear
Programming and Probabilistic Approaches}

\author{Lei Yu\footnote{This work  was supported by the NSFC grant  62101286 and the Fundamental Research Funds for the Central Universities of China (Nankai University).}}

\address{School of Statistics and Data Science, LPMC \& KLMDASR, \\
 Nankai University, \\
 Tianjin 300071, China \\
 Email: leiyu@nankai.edu.cn. }

\begin{abstract}
Let $Q_{n}^{r}$ be the graph with vertex set $\{-1,1\}^{n}$ in which
two vertices are joined if their Hamming distance is at most $r$.
The edge-isoperimetric problem for $Q_{n}^{r}$ is that: For every
$(n,r,M)$ such that $1\le r\le n$ and $1\le M\le2^{n}$, determine
the minimum edge-boundary size of a subset of vertices of $Q_{n}^{r}$
with a given size $M$. In this paper, we apply two different approaches
to prove bounds for this problem. The first approach is a linear programming
approach and the second is a probabilistic approach. Our bound derived
by the first approach generalizes the tight bound for $M=2^{n-1}$
derived by Kahn, Kalai, and Linial in 1989. Moreover, our bound is
also tight for $M=2^{n-2}$ and $r\le\frac{n}{2}-1$. Our bounds derived
by the second approach are expressed in terms of the \emph{noise stability},
and they are shown to be asymptotically tight as $n\to\infty$ when $r=2\lfloor\frac{\beta n}{2}\rfloor+1$
and $M=\lfloor\alpha2^{n}\rfloor$ for fixed $\alpha,\beta\in(0,1)$,
and is tight up to a factor $2$ when $r=2\lfloor\frac{\beta n}{2}\rfloor$
and $M=\lfloor\alpha2^{n}\rfloor$. In fact, the edge-isoperimetric
problem is equivalent to a ball-noise stability problem which is a
variant of the traditional (i.i.d.-) noise stability problem. Our
results can be interpreted as bounds for the ball-noise stability
problem. 
\end{abstract}
\maketitle

\section{\label{sec:Introduction}Introduction}

The isoperimetric problem is one of most classic problems, which is
to determine the minimum possible boundary-size (i.e., perimeter)
of a set with a fixed size (i.e., volume). A famous result for the
isoperimetric problem in the $n$-dimensional Euclidean space states
that an $n$-ball has the smallest surface area per given volume.
In last several decades, an analogue of the isoperimetric problem
was considered in the discrete setting. Let $G=(V,E)$ be a graph
and $A\subseteq V$ a non-empty subset of vertices of $G$. The edge-boundary
$\partial A$ of $A$ is the set of all edges of $G$ joining a vertex
in $A$ to a vertex in $V\backslash A$. The edge-isoperimetric problem
for $G$ asks for the determination of 
\begin{equation}
\min\{|\partial A|:A\subseteq V,|A|=M\},\label{eq:minboundary}
\end{equation}
for each integer $M$. When the graph $G$ is set to (the powers of)
discrete hypercubes, the corresponding isoperimetric problem attracts
a lot of attentions due to its importance to related problems in combinatorics,
discrete probability, computer science, social choice theory, and
others; see e.g. \cite{O'Donnell14analysisof,kahn1988influence,harper1991problem,bezrukov1999edge,benjamini1999noise}.
For the hypercube\footnote{Without loss of generality, one can also consider the hypercube as
$\{0,1\}^{n}$, and our results can be easily converted into this
case via a simple bijection $x\in\{0,1\}\mapsto(-1)^{x}$. We choose
``$\{-1,1\}^{n}$'' since the Fourier transform of a function on
this set is easier to present. } $\{-1,1\}^{n}$, the Hamming distance $d_{\mathrm{H}}(\mathbf{x},\mathbf{y}):=|\{i:\:x_{i}\neq y{}_{i}\}|$
between two vectors $\mathbf{x}$ and $\mathbf{y}$ in $\{-1,1\}^{n}$
is defined as the number of coordinates in which they differ. For
positive integers $n$ and $r$ such that $r\le n$, we let $Q_{n}^{r}$
denote the $r$-th power of the $n$-dimensional discrete hypercube
graph, i.e., the graph with vertex-set $\{-1,1\}^{n}$ in which two
vectors are joined if they are Hamming distance at most $r$ apart.
When $r=1$, the hypercube graph $Q_{n}^{r}$ is denoted as $Q_{n}$
for brevity. The edge-isoperimetric problem for $Q_{n}^{r}$ is hence
formulated as follows. For every $(n,r,M)$ such that\footnote{Throughout this paper, we denote $[a:b]$ for $a,b\in\mathbb{R}$
such that $a\le b$ as the set of integers between $a$ and $b$ (i.e.,
$[a,b]\cap\mathbb{Z}$). .} $r\in[1:n]$ and $M\in[1:2^{n}]$, determine the minimum boundary-size
of a subset (also termed a \emph{code}) of $Q_{n}^{r}$ with a given
size $M$. Throughout this paper, we denote the (normalized) volume
as 
\begin{align*}
\alpha & :=\frac{M}{2^{n}}\textrm{ and }\beta:=\frac{r}{n}.
\end{align*}
The edge-isoperimetric problem is also related to the estimate of
distance distribution of a subset in the hypercube $Q_{n}$. For a
graph $G=(V,E)$ and a non-empty subset $A\subseteq V$, the subgraph
induced by $A$ is denoted as $G[A]$, which is the graph whose vertex
set is $A$ and whose edge set consists of all of the edges in $E$
that have both endpoints in $A$. Let $e(A)$ denote the number of
edges of $G[A]$. Indeed, if $G$ is a $d$-regular graph, then $2e(A)+|\partial A|=d|A|$
for all $A\subseteq V$. Denote $B_{r}^{(n)}:=\{\mathbf{x}:d_{\mathrm{H}}(\mathbf{x},\mathbf{1})\le r\}$
(or shortly $B_{r}$) as the $r$-radius ball with center $\mathbf{1}=\{1,1,...,1\}$.
Denote the cardinality of $B_{r}^{(n)}$ as ${n \choose \le r}:=\sum_{i=0}^{r}{n \choose i}$.
Similarly, we denote the Hamming sphere with the same radius and center
as $S_{r}^{(n)}:=\{\mathbf{x}:d_{\mathrm{H}}(\mathbf{x},\mathbf{1})=r\}$
(or shortly $S_{r}$) and its cardinality as ${n \choose r}$. Since
$Q_{n}^{r}$ is $({n \choose \le r}-1)$-regular, it holds that for
$A\subseteq\{-1,1\}^{n}$ with size $M$, $2e(A)+|\partial A|=[{n \choose \le r}-1]M.$
Hence, the edge-isoperimetric problem is equivalent to determining
\[
\max\{e(A):A\subseteq\{-1,1\}^{n},|A|=M\}.
\]
For a non-empty subset $A\subseteq\{-1,1\}^{n}$, the \emph{distance
distribution} of $A$ is defined as the following probability mass
function\footnote{Our definition of the distance distribution is slightly different
from the classic one (e.g., in \cite{macwilliams1977theory}), since
in the classic definition, the factor is $\frac{1}{|A|}$, rather
than $\frac{1}{|A|^{2}}$. Our choice is more convenient when the
code size is large (e.g., linear in $2^{n}$). }: 
\[
P^{(A)}(i):=\frac{1}{|A|^{2}}|\{(\mathbf{x},\mathbf{x}')\in A^{2}:d_{\mathrm{H}}(\mathbf{x},\mathbf{x}')=i\}|,\;i\in[0:n].
\]
It is clear that $P^{(A)}(0)=\frac{1}{|A|}$, $\sum_{i=0}^{n}P^{(A)}(i)=1$,
and $P^{(A)}(i)\ge0$ for $i\in[0:n]$. Furthermore, by definition,
if $|A|=M$ then $e(A)=\frac{M^{2}}{2}\sum_{i=1}^{r}P^{(A)}(i).$
Hence, the edge-isoperimetric problem is also equivalent to determining
\[
\max\{\sum_{i=0}^{r}P^{(A)}(i):A\subseteq\{-1,1\}^{n},|A|=M\},
\]
i.e., the estimation of the cumulative distribution function of the
distance distribution.

A trivial upper bound for the edge-isoperimetric problem is $\sum_{i=0}^{r}P^{(A)}(i)\le1$,
which is attained if $M$ is small enough, more precisely, if $M\le{n \choose \le b/2}$
(i.e., the optimal sets $A$ are contained in a Hamming ball of radius
$b/2$) \cite{kleitman1966combinatorial,kahn1988influence}. For
$r=1$ (i.e., for $Q_{n}$), the edge-isoperimetric problem was solved
by Harper \cite{harper1964optimal}, Lindsey \cite{lindsey1964assignment},
Bernstein \cite{bernstein1967maximally} and Hart \cite{hart1976note},
who showed that \emph{lexicographic subsets} are optimal in minimizing
the edge-boundary size. Here, lexicographic subsets are subsets whose
elements are given by initial segments of the lexicographic ordering
on $\{-1,1\}^{n}$. In fact, lexicographic subsets are generalizations
of subcubes, and reduce to subcubes when the sizes of them are $2^{k}$
for integers $k$. Furthermore, for $r\ge2$ and $M=2^{n-1}$, the
edge-isoperimetric problem was solved by Kahn, Kalai, and Linial \cite{kahn1988influence}
in 1989, who showed that subcubes are also optimal for this case.
However, the problem for $r\ge2$ and $M\neq2^{n-1}$ has remained
open. In this paper, we make progress on these unsolved cases, more
specifically, on the cases of $r\ge2$ and $M=\alpha2^{n}$ with $\alpha\in(0,\frac{1}{2})$.
In other words, the case in which $M$ is linear in $2^{n}$ is considered
in this paper. For this case, we provide two bounds for the edge-isoperimetric
problem. In particular, we prove that subcubes are also optimal for
$r\ge2$ and $M=2^{n-2}$.

When $M$ is exponential in $n$, by using an improved hypercontractivity
inequality, Kirshner and Samorodnitsky \cite{kirshner2021moment}
recently showed that for $M=2^{nH(\sigma)}$ with $\sigma\in(0,\frac{1}{2})$
and for $i=n\lambda$ with $\lambda\le2\sigma(1-\sigma)$, 
\begin{equation}
P^{(A)}(i)\leq2^{n[\sigma H(\frac{\lambda}{2\sigma})+(1-\sigma)H(\frac{\lambda}{2(1-\sigma)})-H(\sigma)]},\label{eq:-12}
\end{equation}
where $H(p):=-p\log_{2}p-(1-p)\log_{2}(1-p)$ for $p\in[0,1]$ denotes
the binary entropy (here, the convention $0\log_{2}0=0$ is adopted).
By computing the derivative, one can find that given $\sigma$, the
exponent at the right side of \eqref{eq:-12} is non-decreasing in
$\lambda$ for $\lambda\le2\sigma(1-\sigma)$. Hence, the inequality
in \eqref{eq:-12} implies that for $M=2^{nH(\sigma)}$ with $\sigma\in(0,\frac{1}{2})$
and for $r=n\beta$ with $\beta\le2\sigma(1-\sigma)$, 
\begin{align}
\sum_{i=0}^{r}P^{(A)}(i) & \leq(r+1)2^{n[\sigma H(\frac{\beta}{2\sigma})+(1-\sigma)H(\frac{\beta}{2(1-\sigma)})-H(\sigma)]}.\label{eq:-1}
\end{align}
 {When $r_{n}=\lfloor\beta n\rfloor$ for a fixed $\beta\in(0,1)$
and let $n\to\infty$, it holds that 
\begin{align}
 & \lim_{n\to\infty}-\frac{1}{n}\log_{2}\max_{A:|A|\le2^{nH(\sigma)}}\sum_{i=0}^{r_{n}}P^{(A)}(i)\nonumber \\
 & \qquad=\begin{cases}
H(\sigma)-\sigma H(\frac{\beta}{2\sigma})-(1-\sigma)H(\frac{\beta}{2(1-\sigma)}), & \beta\le2\sigma(1-\sigma)\\
0, & \beta>2\sigma(1-\sigma)
\end{cases}.\label{eq:-15}
\end{align}
Here, the optimal exponent in \eqref{eq:-15} is attained by a sequence
of Hamming balls with radii (approximately) equal to $n\sigma$. The
first clause at the right side of \eqref{eq:-15} follows from \eqref{eq:-1},
and the second one follows by the following facts: 1. when $\beta=2\sigma(1-\sigma)$,
the right side of \eqref{eq:-15} vanishes; 2. given a set $A$, $r\mapsto\sum_{i=0}^{r}P^{(A)}(i)$
is non-decreasing, which implies that the left side of \eqref{eq:-15}
is non-increasing in $\beta$; 3. the left side of \eqref{eq:-15}
is non-negative. It is worth noting that if we replace $\sum_{i=0}^{r_{n}}P^{(A)}(i)$
at the left side of \eqref{eq:-15} with $P^{(A)}(r_{n})$, then the
asymptotic exponent is different from the above. Specifically, this
new exponent is zero for $2\sigma(1-\sigma)\le\beta\le\frac{1}{2}$
and it is symmetric with respect to $\beta=\frac{1}{2}$; see details
in Remark 29 of \cite{kirshner2021moment}. }Furthermore, Rashtchian
and Raynaud \cite{rashtchian2019edge} also derived different bounds
for the edge-isoperimetric problem for $Q_{n}^{r}$. Their bounds
are tight up to a factor of $\exp(\Theta(r))$ (i.e., a factor depending
only upon $r$). 

\subsection{Ball-Noise Stability: Probabilistic Reformulation of the Edge-Isoperimetric
Problem}

In this subsection, we reformulate the edge-isoperimetric problem
in probabilistic language. Let $\mathbf{X}\sim\mathrm{Unif}\{-1,1\}^{n}$
and $\mathbf{Y}=\mathbf{X}\circ\mathbf{Z}=(X_{i}\cdot Z_{i})_{1\le i\le n}$
where $\mathbf{Z}\in\{-1,1\}^{n}$ is independent of $\mathbf{X}$
and $\circ$ denotes the Hadamard product (element-wise product). 
\begin{defn}
\label{def:-For-two-1-1} For $f:\{-1,1\}^{n}\to\mathbb{R}$ and $r\in[0:n]$,
the \emph{sphere-noise stability} and \emph{ball-noise stability}
of $f$ at $r$ are respectively 
\begin{equation}
\mathbf{SStab}_{r}[f]:=\mathbb{E}[f(\mathbf{X})f(\mathbf{Y})]=\sum_{\mathbf{x},\mathbf{y}\in f^{-1}(1)}\frac{1\{d_{\mathrm{H}}(\mathbf{x},\mathbf{y})=r\}}{2^{n}{n \choose r}}\label{eq:-28}
\end{equation}
with $\mathbf{X}\sim\mathrm{Unif}\{-1,1\}^{n},\mathbf{Z}\sim\mathrm{Unif}(S_{r})$,
and 
\[
\mathbf{BStab}_{r}[f]:=\mathbb{E}[f(\mathbf{X})f(\mathbf{Y})]=\sum_{\mathbf{x},\mathbf{y}\in f^{-1}(1)}\frac{1\{d_{\mathrm{H}}(\mathbf{x},\mathbf{y})\le r\}}{2^{n}{n \choose \le r}}
\]
with $\mathbf{X}\sim\mathrm{Unif}\{-1,1\}^{n},\mathbf{Z}\sim\mathrm{Unif}(B_{r})$. 
\end{defn}

Obviously, the joint distribution $P_{\mathbf{XY}}$ is symmetric
(i.e., $P_{\mathbf{XY}}=P_{\mathbf{YX}}$) if $\mathbf{X}\sim\mathrm{Unif}\{-1,1\}^{n}$
and $\mathbf{Z}\sim\mathrm{Unif}(S_{r})$ or $\mathbf{Z}\sim\mathrm{Unif}(B_{r})$.
The edge-isoperimetric problem for $Q_{n}^{r}$ is equivalent to the
following question: For every $(n,r,M)$ such that $1\le r\le n$
and $1\le M\le2^{n}$, determine 
\begin{align}
\Gamma_{\mathrm{S}}^{(n)}(M,r) & :=\max_{\substack{f:\{-1,1\}^{n}\to\{0,1\}\\
\mathbb{P}[f=1]=\alpha
}
}\mathbf{SStab}_{r}[f]\label{eq:-101}\\
\Gamma_{\mathrm{B}}^{(n)}(M,r) & :=\max_{\substack{f:\{-1,1\}^{n}\to\{0,1\}\\
\mathbb{P}[f=1]=\alpha
}
}\mathbf{BStab}_{r}[f]\label{eq:-102}
\end{align}
and their limits as $n\to\infty$ for fixed $\alpha=\frac{M}{2^{n}}$
and $\beta=\frac{r}{n}$. Here, the Boolean function $f:\{-1,1\}^{n}\to\{0,1\}$
in the optimizations can be seen as the indicator of the code $A$.
Our motivation for studying the case in which $M$ is linear in $2^{n}$
comes from this probabilistic formulation of the edge-isoperimetric
problem, since in this probabilistic formulation, $\alpha=\frac{M}{2^{n}}$
corresponds to the probability of the code $A$ under the uniform
measure. For even and odd values of $r$, the limit behaviour is in
fact different for the case of interest here. For $\alpha,\beta\in(0,1)$,
we define 
\begin{align}
\Gamma_{\mathrm{odd,S}}(\alpha,\beta) & :=\lim_{n\to\infty}\Gamma_{\mathrm{S}}^{(n)}(\lfloor\alpha2^{n}\rfloor,2\lfloor\frac{\beta n}{2}\rfloor+1)\nonumber \\
\overline{\Gamma}_{\mathrm{even,S}}(\alpha,\beta) & :=\limsup_{n\to\infty}\Gamma_{\mathrm{S}}^{(n)}(\lfloor\alpha2^{n}\rfloor,2\lfloor\frac{\beta n}{2}\rfloor)\nonumber \\
\underline{\Gamma}_{\mathrm{even,S}}(\alpha,\beta) & :=\liminf_{n\to\infty}\Gamma_{\mathrm{S}}^{(n)}(\lfloor\alpha2^{n}\rfloor,2\lfloor\frac{\beta n}{2}\rfloor).\label{eq:-107}
\end{align}
By replacing sphere noise with ball noise, $\Gamma_{\mathrm{odd,B}}(\alpha,\beta),\overline{\Gamma}_{\mathrm{even,B}}(\alpha,\beta),\underline{\Gamma}_{\mathrm{even,B}}(\alpha,\beta)$
are defined similarly. (The limits in the definitions of $\Gamma_{\mathrm{odd,S}}(\alpha,\beta)$
and $\Gamma_{\mathrm{odd,B}}(\alpha,\beta)$ exist, which will be
shown in Theorem \ref{thm:bsiid}.) We term the optimization problems
in \eqref{eq:-101} and \eqref{eq:-102} respectively as the \emph{sphere-noise
stability} and \emph{ball-noise stability} problems.

The sphere-noise stability and ball-noise stability problems can be
seen as variants of the traditional i.i.d.-noise stability problem.
In the traditional noise stability problem, $\mathbf{Z}\sim\mathrm{Rad}^{\otimes n}(\beta)$.
Here $\mathrm{Rad}^{\otimes n}(\beta)$ denotes the $n$-product of
the biased Rademacher distribution $\mathrm{Rad}(\beta),\beta\in(0,\frac{1}{2})$
with itself, where the biased Rademacher distribution $\mathrm{Rad}(\beta)$
is a distribution having the probability mass function 
\[
P_{Z}(z)=\begin{cases}
1-\beta & z=1\\
\beta & z=-1
\end{cases}.
\]
The \emph{noise stability }of a function $f$ is defined as 
\begin{equation}
\mathbf{Stab}_{\beta}[f]:=\mathbb{E}[f(\mathbf{X})f(\mathbf{Y})]=\sum_{\mathbf{x},\mathbf{y}\in f^{-1}(1)}\frac{\beta^{d_{\mathrm{H}}(\mathbf{x},\mathbf{y})}(1-\beta)^{n-d_{\mathrm{H}}(\mathbf{x},\mathbf{y})}}{2^{n}},\label{eq:-11}
\end{equation}
where $\mathbf{X}\sim\mathrm{Unif}\{-1,1\}^{n},\mathbf{Z}\sim\mathrm{Rad}^{\otimes n}(\beta)$,
and again $\mathbf{Y}=\mathbf{X}\circ\mathbf{Z}$. Similarly to \eqref{eq:-101}-\eqref{eq:-107},
define 
\begin{align}
\Gamma_{\mathrm{IID}}^{(n)}(M,\beta):=\max_{\substack{f:\{-1,1\}^{n}\to\{0,1\}\\
\mathbb{P}[f=1]=\alpha
}
}\mathbf{Stab}_{\beta}[f]\quad\textrm{ and }\quad\Gamma_{\mathrm{IID}}(\alpha,\beta) & :=\lim_{n\to\infty}\Gamma_{\mathrm{IID}}^{(n)}(\lfloor\alpha2^{n}\rfloor,\beta).\label{eq:GammaIID}
\end{align}
Obviously, the limit in \eqref{eq:GammaIID} exists, since $\Gamma_{\mathrm{IID}}^{(n)}(\lfloor\alpha2^{n}\rfloor,\beta)$
is non-decreasing in $n$ for given $\alpha,\beta$.

The edge-isoperimetric problem for $Q_{n}^{r}$ and the ball-noise
stability problem in \eqref{eq:-102} are equivalent, as shown in
the following proposition. 
\begin{prop}
\label{prop:}For $A\subseteq\{-1,1\}^{n}$ with size $\alpha2^{n}$,
$\mathbf{BStab}_{r}[1_{A}]=\frac{\alpha^{2}2^{n}}{{n \choose \le r}}\sum_{i=0}^{r}P^{(A)}(i).$ 
\end{prop}

\begin{proof}
\begin{align*}
\mathbf{BStab}_{r}[1_{A}] & =\mathbb{P}[\mathbf{X}\in A,\mathbf{Y}\in A]=\sum_{\mathbf{x},\mathbf{y}\in A}\frac{1\{d_{\mathrm{H}}(\mathbf{x},\mathbf{y})\le r\}}{2^{n}{n \choose \le r}}\\
 & =\frac{\alpha^{2}2^{n}}{{n \choose \le r}}\sum_{\mathbf{x},\mathbf{y}\in A}\frac{1\{d_{\mathrm{H}}(\mathbf{x},\mathbf{y})\le r\}}{(\alpha2^{n})^{2}}=\frac{\alpha^{2}2^{n}}{{n \choose \le r}}\sum_{i=0}^{r}P^{(A)}(i).
\end{align*}
\end{proof}
 {Combining \eqref{eq:-15} and Proposition \ref{prop:}
yields that when}\footnote{ {Throughout this paper, we use $o_{n}(1)$ to denote
a term (or sequence) that vanishes as $n\to\infty$.}} {{} $M_{n}=2^{n(H(\sigma)+o_{n}(1))}$ (i.e., $\alpha_{n}=2^{n(H(\sigma)-1+o_{n}(1))}$)
and $r_{n}=\lfloor\beta n\rfloor$ for some fixed $\sigma\in(0,\frac{1}{2}),\beta\in(0,1)$,
it holds that 
\begin{align}
 & \lim_{n\to\infty}-\frac{1}{n}\log_{2}\Gamma_{\mathrm{B}}^{(n)}(M_{n},r_{n})\nonumber \\
 & \qquad=\begin{cases}
D((1-\sigma-\frac{\beta}{2},\frac{\beta}{2},\frac{\beta}{2},\sigma-\frac{\beta}{2})\|(\frac{1-\beta}{2},\frac{\beta}{2},\frac{\beta}{2},\frac{1-\beta}{2})), & \beta\le2\sigma(1-\sigma)\\
1+H(\min\{\beta,\frac{1}{2}\})-2H(\sigma), & \beta>2\sigma(1-\sigma)
\end{cases},\label{eq:-14}
\end{align}
}where $D(Q\|P):=\sum_{x}Q(x)\log_{2}\frac{Q(x)}{P(x)}$ denotes the
relative entropy between two distributions $Q$ and $P$.

In the literature, the (i.i.d.) noise stability problem was studied
by Benjamini, Kalai, and Schramm \cite{benjamini1999noise}. When
$\alpha=\frac{1}{2}$ and $n\ge1$, $\Gamma_{\mathrm{IID}}^{(n)}(2^{n-1},\beta)=\frac{1-\beta}{2}$.
This is a consequence of Witsenhausen's results on maximal correlation
\cite{witsenhausen1975sequences}. When $\alpha=\frac{1}{4}$ and
$n\ge2$, Yu and Tan \cite{yu2021non} showed that $\Gamma_{\mathrm{IID}}^{(n)}(2^{n-2},\beta)=(\frac{1-\beta}{2})^{2}$.
In the literature, hypercontractivity inequalities were also used
to prove asymptotically tight (up to a factor $(\log\frac{1}{\alpha})^{k}$
for some $k$) bounds on $\Gamma_{\mathrm{IID}}(\alpha,\beta)$ as
$\alpha\to0$ for fixed $\beta$; see \cite{O'Donnell14analysisof,kamath2016non,yu2021non}.
Kirshner and Samorodnitsky's improved hypercontractivity inequality
in Theorem 8 of \cite{kirshner2021moment} implies that when $M_{n}=2^{n(H(\sigma)+o_{n}(1))}$
(i.e., $\alpha=2^{n(H(\sigma)-1+o_{n}(1))}$) for some $\sigma\in(0,\frac{1}{2})$,
the exponent of $\Gamma_{\mathrm{IID}}^{(n)}(M_{n},\beta)$ for $\beta\in(0,\frac{1}{2})$
is 
\begin{align}
\lim_{n\to\infty}-\frac{1}{n}\log_{2}\Gamma_{\mathrm{IID}}^{(n)}(M_{n},\beta) & =\min_{Q_{XY}:Q_{X}=Q_{Y}=(1-\sigma,\sigma)}D(Q_{XY}\|P_{XY})\label{eq:}\\
 & =\min_{0\le\theta\le2\sigma}D\big((1-\sigma-\frac{\theta}{2},\frac{\theta}{2},\frac{\theta}{2},\sigma-\frac{\theta}{2})\|(\frac{1-\beta}{2},\frac{\beta}{2},\frac{\beta}{2},\frac{1-\beta}{2})\big),\label{eq:-16}
\end{align}
 {where the right side of \eqref{eq:} is termed the} {\emph{
minimum relative entropy}} {{} over couplings of $Q_{X}=Q_{Y}=(1-\sigma,\sigma)$
\cite{yu2021graphs,yu2021strong}, and the unique optimal $\theta$
attaining the minimum in \eqref{eq:-16} is $\theta^{*}=\frac{\sqrt{1+4(\kappa-1)\sigma(1-\sigma)}-1}{\kappa-1}$
with $\kappa=(\frac{1-\beta}{\beta})^{2}$. Here, the optimal exponent
in \eqref{eq:-16} is attained by a sequence of Hamming balls (or
Hamming spheres). The result in \eqref{eq:} and \eqref{eq:-16} was
generalized to the two set version of noise stability and also generalized
to arbitrary distributions on finite alphabets or Polish spaces in
\cite{yu2021graphs,yu2021strong}. Furthermore, for $\sigma<\frac{1}{2},\beta\le2\sigma(1-\sigma)$,
it holds that $\theta^{*}<\beta$, which implies that for this case,
the expression in \eqref{eq:-16} is strictly smaller than that in
\eqref{eq:-14}. }

\subsection{Main Results}

In this paper, we study the discrete edge-isoperimetric problem for
$Q_{n}^{r}$ with $r\ge1$. We apply two different techniques to derive
bounds for this problem. The first one is Fourier analysis combined
with linear programming duality. By such a technique, we prove the
following bound which is called linear programming (LP) bound. 
\begin{thm}[Linear Programming Bounds]
\label{thm:LPBound} 1. For $\alpha:=\frac{M}{2^{n}}\le\frac{1}{2}$,
\begin{align}
\Gamma_{\mathrm{B}}^{(n)}(M,r) & \leq\alpha^{2}[1-\frac{\psi_{n}^{+}(\alpha,r)}{{n \choose \le r}}]+\alpha(1-\alpha)\frac{{n-1 \choose r}}{{n \choose \le r}},\label{eq:-5}
\end{align}
where\footnote{Throughout this paper, $[x]^{+}:=\max\{x,0\}$. } $\psi_{n}^{+}(\alpha,r):=[\psi_{n}(\alpha,r)]^{+}$
and 
\begin{equation}
\psi_{n}(\alpha,r):=\begin{cases}
\underset{\textrm{odd }k\in[n-\tau(n):n]}{\max}\frac{n{n-2 \choose r-1}}{2k-n}[2(\frac{1}{\alpha}-1)-\frac{1}{\alpha}\frac{n+1}{k+1}], & \textrm{even }r\le n/2-1\\
\underset{k\in[n-\tau(n):n]\cap F}{\max}\frac{(n-1){n-2 \choose r-1}}{2k-n-1}[2(\frac{1}{\alpha}-1)-\frac{1}{\alpha}\frac{n}{k}], & \textrm{odd }r\le n/2-1\\
\underset{\textrm{odd }k\in[\frac{n}{2}+1:n]}{\max}\frac{k{n-2 \choose r+1}}{(2k-n)}[2(\frac{1}{\alpha}-1)-\frac{1}{\alpha}\frac{n}{k}], & \textrm{even }r>n/2-1\\
\underset{\textrm{odd }k\in[\frac{n}{2}+1:n]}{\max}\frac{k{n-2 \choose r}}{(2k-n)}[2(\frac{1}{\alpha}-1)-\frac{1}{\alpha}\frac{n}{k}], & \textrm{odd }r>n/2-1
\end{cases}\label{eq:psi}
\end{equation}
with 
\begin{equation}
\tau(n):=\frac{1}{2}(\frac{n}{2}+2-\sqrt{\frac{n}{2}+2})\label{eq:tau}
\end{equation}
and 
\begin{align}
F & :=\Big\{\textrm{even }k:\;k\ge\frac{n+r+1}{2}\Big\}\nonumber \\
 & \qquad\cup\Big\{\textrm{odd }k:\:k\ge\max\{\frac{n+1}{2},\frac{(n-1)r}{n-1-r}\}+\sqrt{\frac{(n+1)r}{2(n-1-r)}}\Big\}.\label{eq:F}
\end{align}
2. When considering the asymptotic case as $n\to\infty$, we have
that for fixed $\delta>0$, 
\begin{align}
\Gamma_{\mathrm{B}}^{(n)}(M,r) & \leq\alpha^{2}(1-(1-2\beta)\phi_{n}(\alpha,r))+o_{n}(1)\label{eq:-109-1}
\end{align}
holds for all $\alpha:=\frac{M}{2^{n}}\in[\delta,\frac{1}{2}-\delta],\beta:=r/n\in(0,\frac{1}{2}-\delta]\cup[\frac{1}{2},1]$,
where $o_{n}(1)$ is only dependent on $\delta$ and independent of
$M,r$. Here, 
\[
\phi_{n}(\alpha,r):=\begin{cases}
\beta\varphi(\alpha)-(\frac{1}{\alpha}-1), & \textrm{even }r\le n/2-1\\
\beta\hat{\varphi}(\alpha,\beta)-(\frac{1}{\alpha}-1), & \textrm{odd }r\le n/2-1\\
0, & r>n/2-1
\end{cases}
\]
where 
\begin{align*}
\varphi(\alpha) & :=\begin{cases}
\frac{2(1-\sqrt{\alpha})^{2}}{\alpha}, & 0\le\alpha<1/4\\
\frac{1}{\alpha}-2, & 1/4\le\alpha\le1/2
\end{cases}\\
\hat{\varphi}(\alpha,\beta) & :=\begin{cases}
\frac{1}{2\hat{\eta}-1}[2(\frac{1}{\alpha}-1)-\frac{1}{\alpha\hat{\eta}}], & 0\le\alpha<1/4\\
\frac{1}{\alpha}-2, & 1/4\le\alpha\le1/2
\end{cases}
\end{align*}
with $\hat{\eta}:=\max\big\{\frac{1}{2(1-\sqrt{\alpha})},\min\{\frac{1+\beta}{2},\frac{\beta}{1-\beta}\}\big\}.$ 
\end{thm}

For comparison, observe that when $\alpha=2^{-k}$ for a positive
integer $k$, every $(n-k)$-dimensional Hamming subcube $C_{n-k}$
(e.g., $\{1\}^{k}\times\{-1,1\}^{n-k}$) attains the following ball-noise
stability: 
\begin{equation}
\mathbf{BStab}_{r}[1_{C_{n-k}}]=\alpha\frac{{n-k \choose \le r}}{{n \choose \le r}}.\label{eq:-79}
\end{equation}
In particular, for fixed $k$, as $n,r\to\infty$ and $r/n\to\beta$,
\begin{align}
\mathbf{BStab}_{r}[1_{C_{n-k}}] & \to(\frac{1-\beta}{2})^{k}.\label{eq:-114}
\end{align}

\begin{defn}
\label{def:even-part} The \emph{even part} $A_{\mathrm{even}}$ of
a set $A\subseteq\{-1,1\}^{n}$ is defined as $\{\mathbf{x}\in A:d_{\mathrm{H}}(\mathbf{x},\mathbf{1})\textrm{ is even}\}$,
i.e., the intersection of $A$ and the set of vectors $\mathbf{x}$
of even Hamming weight $d_{\mathrm{H}}(\mathbf{x},\mathbf{1})$. The
\emph{odd part} of $A$ is defined as $A_{\mathrm{odd}}=A\backslash A_{\mathrm{even}}$. 
\end{defn}

For even $r$, the same ball-noise stability as in \eqref{eq:-79}
can be also achieved by the even part (or odd part) of an $(n-k+1)$-dimensional
subcube, e.g., $\{\mathbf{x}\in\{1\}^{k-1}\times\{-1,1\}^{n+1-k}:d_{\mathrm{H}}(\mathbf{x},\mathbf{1})\textrm{ is even}\}.$

Comparing Theorem \ref{thm:LPBound} with \eqref{eq:-79} and \eqref{eq:-114},
we know that the bound in \eqref{eq:-5} is tight for $\alpha=1/2$
and $n\ge1$ as well as for $\alpha=1/4$, $r\le n/2-1$, and $n\ge2$.
For the former case ($\alpha=1/2$ and $n\ge1$),  {$\psi_{n}^{+}(1/2,r)=0$},
which leads to the tight result 
\[
\Gamma_{\mathrm{B}}^{(n)}(2^{n-1},r)=\frac{{n-1 \choose \le r}}{2{n \choose \le r}}\textrm{ and }\Gamma_{\mathrm{B}}(1/2,\beta)=\frac{1-\beta}{2},\;\beta\in(0,1).
\]
This recovers a classic result derived by Kahn, Kalai, and Linial
\cite{kahn1988influence}. For the latter case ($\alpha=1/4$, $r\le n/2-1$,
and $n\ge2$),  {$\psi_{n}(1/4,r)=2{n-2 \choose r-1}$, where the optimal
$k$ attaining $\psi_{n}(1/4,r)$ in \eqref{eq:psi} for even $r\le n/2-1$
is $k^{*}=n$ for odd $n$, and $k^{*}=n-1$ for even $n$, and the
optimal $k$ attaining $\psi_{n}(1/4,r)$ for odd $r\le n/2-1$ is
$k^{*}=n$.} This leads to that 
\[
\Gamma_{\mathrm{B}}^{(n)}(2^{n-2},r)=\frac{{n-2 \choose \le r}}{4{n \choose \le r}}\textrm{ and }\Gamma_{\mathrm{B}}(1/4,\beta)=(\frac{1-\beta}{2})^{2},\;\beta\in(0,1/2).
\]
This result is new. 

 {For fixed $\alpha=\frac{M}{2^{n}}\le\frac{1}{2}$
and sufficiently large $n$, Rashtchian and Raynaud's bound in \cite{rashtchian2019edge}
reduces to the following bound: 
\begin{align*}
\Gamma_{\mathrm{B}}^{(n)}(M,r) & \leq\frac{2\alpha}{{n \choose \le r}}[\frac{16\mathrm{e}n}{r}(n-\log_{2}\frac{1}{\alpha})]^{r/2}.
\end{align*}
When considering the setting of interest in this paper (i.e., the
case of fixed $\alpha=\frac{M}{2^{n}}\in(0,\frac{1}{2}]$, $\beta=\frac{r}{n}$,
but $n\to\infty$), Rashtchian and Raynaud's bound becomes $\frac{2\alpha}{{n \choose \le n\beta}}[\frac{16\mathrm{e}}{\beta}(n-\log_{2}\frac{1}{\alpha})]^{n\beta/2}$,
which tends to $\infty$ as $n\to\infty$. Hence, their bound is trivial
for our setting.  We next compare our result with Kirshner and Samorodnitsky's
in \eqref{eq:-12}.  We first restate their result in the language
of the sphere-noise stability. Similarly to Proposition \ref{prop:},
for sphere noise, it holds that $\mathbf{SStab}_{r}[1_{A}]=\frac{\alpha^{2}2^{n}}{{n \choose r}}P^{(A)}(r)$.
Substituting the bound in \eqref{eq:-12} into this formula yields
the following bound: For $M=2^{nH(\sigma)}$ with $\sigma\in(0,\frac{1}{2})$
and $r=n\beta$ with $\beta\le2\sigma(1-\sigma)$, 
\begin{align}
\Gamma_{\mathrm{S}}^{(n)}(M,r) & \leq\frac{1}{{n \choose r}}2^{n[-\beta\log_{2}\beta-(\sigma-\frac{\beta}{2})\log_{2}(2\sigma-\beta)-(1-\sigma-\frac{\beta}{2})\log_{2}(2-2\sigma-\beta)]}.\label{eq:-33}
\end{align}
For fixed $0<\alpha\le\frac{1}{2}$, by solving $2^{nH(\sigma)}=\alpha2^{n}$,
we obtain $\sigma=\frac{1-t}{2}$ where $t=2\sqrt{\frac{1}{2n}\ln\frac{1}{\alpha}}+o_{n}(\frac{1}{\sqrt{n}})$.
On the other hand, for fixed $0<\beta\le\frac{1}{2}$, Stirling's
formula implies ${n \choose r}=2^{nH(\beta)-\frac{1}{2}\log_{2}(2\pi n)+O_{n}(1)}.$
By this formula, the right side of \eqref{eq:-33} reduces to $2^{-\frac{2}{1-\beta}\log_{2}\frac{1}{\alpha}+\frac{1}{2}\log_{2}(2\pi n)+O(1)}$,
which obviously also tends to $\infty$ as $n\to\infty$. In other
words, Kirshner and Samorodnitsky's bound becomes trivial as well
for our setting.  However, it should be noted that Kirshner and
Samorodnitsky's bound outperforms our bound when $\alpha$ vanishes
exponentially as $n\to\infty$, since their bound is exponentially
tight in this setting. }

The proof of Theorem \ref{thm:LPBound} is given in Section \ref{sec:Proof-of-Theorem-LPB}.
Here we provide the outline of the proof. In our proof, we first relax
the edge-isoperimetric problem to a linear program by employing Fourier
analysis. By duality in linear programming, we then rewrite this program
as its dual. Finally, we find a feasible solution to the dual program
which hence provides a lower bound for the primal program. Such a
lower bound also results in a lower bound for the edge-isoperimetric
problem.

We next present our second bound, which is proven by a probabilistic
approach. 
\begin{thm}[Probabilistic Bounds]
\label{thm:bsiid} For $0<\alpha,\beta\leq\frac{1}{2}$, 
\begin{align}
 & \Gamma_{\mathrm{odd,S}}(\alpha,\beta)=\Gamma_{\mathrm{odd,B}}(\alpha,\beta)=\Gamma_{\mathrm{IID}}(\alpha,\beta),\label{eq:-13}\\
\Gamma_{\mathrm{IID}}(\alpha,\beta),\,\frac{1}{2}\Gamma_{\mathrm{IID}}(2\alpha,\beta)\leq & \underline{\Gamma}_{\mathrm{even,S}}(\alpha,\beta)\leq\overline{\Gamma}_{\mathrm{even,S}}(\alpha,\beta)\leq2\Gamma_{\mathrm{IID}}(\alpha,\beta),\label{eq:-116}\\
\Gamma_{\mathrm{IID}}(\alpha,\beta)\leq & \underline{\Gamma}_{\mathrm{even,B}}(\alpha,\beta)\leq\overline{\Gamma}_{\mathrm{even,B}}(\alpha,\beta)\leq2(1-\beta)\Gamma_{\mathrm{IID}}(\alpha,\beta).\label{eq:-117}
\end{align}
\end{thm}

We conjecture that for $0<\alpha,\beta\leq\frac{1}{2}$, $\underline{\Gamma}_{\mathrm{even,S}}(\alpha,\beta)=\overline{\Gamma}_{\mathrm{even,S}}(\alpha,\beta)=\frac{1}{2}\Gamma_{\mathrm{IID}}(2\alpha,\beta)$
and $\underline{\Gamma}_{\mathrm{even,B}}(\alpha,\beta)=\overline{\Gamma}_{\mathrm{even,B}}(\alpha,\beta)=\Gamma_{\mathrm{IID}}(\alpha,\beta)$.
This is true if $\overline{\Gamma}_{\mathrm{even,S}}(\alpha,\beta)$
is attained by $\frac{1+\chi_{[1:n]}}{2}f_{n}$ for some \emph{Fourier-weight
stable} (see Definition \ref{def:Fourierweightstable} in Section
\ref{sec:Proof-of-Theorem-HCB}) sequence of Boolean functions $f_{n}$
and $\overline{\Gamma}_{\mathrm{even,B}}(\alpha,\beta)$ is attained
by some Fourier-weight stable sequence of Boolean functions $g_{n}$.
 For the ball noise case, this conjecture was confirmed positively
for $\alpha=1/2$ by combining Kahn, Kalai, and Linial's result  \cite{kahn1988influence}
and Witsenhausen's result \cite{witsenhausen1975sequences}, and
for $\alpha=1/4$ by combining Yu and Tan's result \cite{yu2021non}
and Theorem \ref{thm:LPBound} above. 

The proof of Theorem \ref{thm:bsiid} is given in Section \ref{sec:Proof-of-Theorem-HCB}.
As observed in Proposition \ref{prop:}, the edge-isoperimetric problem
for $Q_{n}^{r}$ is equivalent to a ball-noise stability problem.
By Fourier analysis, we show that this ball-noise stability and the
sphere-noise stability are bounded by the traditional i.i.d. noise
stability. Hence, we obtain lower bounds for the edge-isoperimetric
problem.

 {Instead of the noise stability problem, a hypercontractivity
inequality for the sphere noise was recently investigated by Polyanskiy
\cite{polyanskiy2019hypercontractivity}. Polyanskiy's hypercontractivity
inequality differs from the sphere-noise stability problem here in
two aspects: 1. The hypercontractivity inequality concerns estimations
of $\mathbb{E}[f(\mathbf{X})f(\mathbf{Y})]$ for all possible real-valued
or complex-valued functions $f$, while the noise stability problem
restricts $f$ to be a Boolean (i.e., binary-valued) function; 2.
the noise model $(\mathbf{X},\mathbf{Y})$ in Polyanskiy's hypercontractivity
inequality is formed by adding the sphere noise twice to the uniform
random variable $\mathbf{X}$, i.e., $\mathbf{Y}=\mathbf{X}\circ\mathbf{Z}_{1}\circ\mathbf{Z}_{2}$
where $\mathbf{Z}_{1}$ and $\mathbf{Z}_{2}$ are i.i.d. sphere noises,
while the noise model $(\mathbf{X},\mathbf{Y})$ in the sphere-noise
stability problem here is formed by adding the sphere noise only once,
i.e., $\mathbf{Y}=\mathbf{X}\circ\mathbf{Z}$ where $\mathbf{Z}$
is a sphere noise. Note that for independent Rademacher noises $\mathbf{Z}_{1}$
and $\mathbf{Z}_{2}$, the Hadamard product $\mathbf{Z}_{1}\circ\mathbf{Z}_{2}$
is still a Rademacher noise. However, this is not true for sphere
noise. Even so,  our strategy to prove Theorem \ref{thm:bsiid} is
similar to the one used by Polyanskiy \cite{polyanskiy2019hypercontractivity},
and both of them exploit the connections between the i.i.d. Rademacher
noise and the sphere-noise (or ball-noise).  More precisely, the
Rademacher noise $\mathbf{Z}\sim\mathrm{Rad}^{\otimes n}(\beta)$
in fact can be regarded as a ``smoothing'' version (weighted sum,
or \textquoteleft binomial analogue\textquoteright ) of the sphere-noise
(compare the expressions in \eqref{eq:-28} and \eqref{eq:-11}),
and moreover, the corruption effect by the Rademacher noise is mainly
determined by the component (the sphere-noise) of radius around $n\beta$
(which is well-known in information theory). }

The quantity $\Gamma_{\mathrm{IID}}(\alpha,\beta)$ was widely studied
in the literature. Define $\rho:=1-2\beta$ and $\Lambda_{\rho}(\alpha)$
as the Gaussian quadrant probability defined by $\Lambda_{\rho}(\alpha)=\mathbb{P}[Z_{1}>t,Z_{2}>t]$,
where $Z_{1},Z_{2}$ are joint standard Gaussians with correlation
$\mathbb{E}[Z_{1}Z_{2}]=\rho$ and $t$ is a real number such that
$\mathbb{P}[Z_{1}>t]=\alpha$. The small-set expansion theorem on
\citep[p. 264]{O'Donnell14analysisof} states that 
\begin{align}
\Gamma_{\mathrm{IID}}(\alpha,\beta) & \le\alpha^{\frac{1}{1-\beta}}.\label{eq:-108}
\end{align}
On the other hand, for all $0<\alpha\le\frac{1}{2}$, Hamming balls
\citep[Exercise 5.32]{O'Donnell14analysisof} yield the lower bound
\begin{align}
\Gamma_{\mathrm{IID}}(\alpha,\beta) & \geq\Lambda_{\rho}(\alpha),\label{eq:-108-1}
\end{align}
and for $\alpha=2^{-k}$ with a positive integer $k$, Hamming subcubes
yield the lower bound
\begin{align}
\Gamma_{\mathrm{IID}}(\alpha,\beta) & \geq(\frac{1-\beta}{2})^{k}.\label{eq:-108-1-1}
\end{align}
 Combining Theorem \ref{thm:bsiid} with \eqref{eq:-108}-\eqref{eq:-108-1-1}
yields the following theorem. 
\begin{thm}[Hypercontractivity Bounds]
\label{thm:HCBound} For $0<\alpha,\beta\le\frac{1}{2}$, 
\begin{align*}
\Lambda_{\rho}(\alpha) & \leq\Gamma_{\mathrm{odd,S}}(\alpha,\beta)=\Gamma_{\mathrm{odd,B}}(\alpha,\beta)\leq\alpha^{\frac{1}{1-\beta}},\\
\frac{1}{2}\Lambda_{\rho}(2\alpha) & \leq\underline{\Gamma}_{\mathrm{even,S}}(\alpha,\beta)\leq\overline{\Gamma}_{\mathrm{even,S}}(\alpha,\beta)\leq2\alpha^{\frac{1}{1-\beta}},\\
\Lambda_{\rho}(\alpha) & \leq\underline{\Gamma}_{\mathrm{even,B}}(\alpha,\beta)\leq\overline{\Gamma}_{\mathrm{even,B}}(\alpha,\beta)\leq2(1-\beta)\alpha^{\frac{1}{1-\beta}}.
\end{align*}
 {Moreover, when $k=\log_{2}\frac{1}{\alpha}$ is a positive integer,
it also holds that 
\begin{align}
\Gamma_{\mathrm{odd,S}}(\alpha,\beta),\;\underline{\Gamma}_{\mathrm{even,B}}(\alpha,\beta) & \ge(\frac{1-\beta}{2})^{k}\quad\textrm{ and }\quad\underline{\Gamma}_{\mathrm{even,S}}(\alpha,\beta)\ge\frac{1}{2}(\frac{1-\beta}{2})^{k-1}.\label{eq:-32}
\end{align}}
\end{thm}

Similarly to the ball-noise case, when $\alpha=2^{-k}$ for a positive
integer $k$, Hamming subcubes $C_{n-k}$ attain the sphere-noise
stability $\alpha{n-k \choose r}/{n \choose r}$. However, for even
$r$, the even part (or odd part) of $C_{n-k+1}$ attains the sphere-noise
stability $\alpha[{n-k \choose r}+{n-k \choose r-1}]/{n \choose r}$,
which is strictly larger than the one by $C_{n-k}$. In particular,
for fixed $k$, as $n,r\to\infty$ and $r/n\to\beta$, it holds that
$\alpha[{n-k \choose r}+{n-k \choose r-1}]/{n \choose r}\to\frac{1}{2}(\frac{1-\beta}{2})^{k-1}$,
i.e., the second lower bound in \eqref{eq:-32}. Similarly, for even
$r$, the lower bound $\frac{1}{2}\Lambda_{\rho}(2\alpha)$ is asymptotically
achieved by the even part (or odd part) of a Hamming ball with volume
$2\alpha$. 

In fact, it was shown in \citep[Exercise 9.24]{O'Donnell14analysisof}
that given $\beta$, $\Lambda_{\rho}(\alpha)=\widetilde{\Theta}(\alpha^{\frac{1}{1-\beta}})$
as $\alpha\to0$. Hence, the bounds in Theorem \ref{thm:HCBound}
are asymptotically tight (up to a factor $(\log\frac{1}{\alpha})^{k}$
for some $k$) as $\alpha\to0$ for fixed $\beta$. By comparing the
LP bounds in Theorem \ref{thm:LPBound} and the hypercontractivity
bounds in Theorem \ref{thm:HCBound}, it is easy to see that the LP
bounds are tighter when $\alpha$ is close to $\frac{1}{2}$ and the
hypercontractivity bounds are tighter when $\alpha$ is close to $0$.

\section{\label{sec:Proof-of-Theorem-LPB}Proof of Theorem \ref{thm:LPBound}}

{\footnotesize{}{}{}}In this section, we apply Fourier analysis
combined with linear programming duality to prove Theorem \ref{thm:LPBound}.
We first introduce Fourier analysis and Krawtchouk polynomials, and
also derive new properties of Krawtchouk polynomials. By using Fourier
analysis, we then relax the edge-isoperimetric problem to a linear
program. By duality in linear programming, we then rewrite this program
as its dual. Finally, we find a feasible solution to the dual program
which hence provides a lower bound for the primal program. Such a
lower bound results in a lower bound for the edge-isoperimetric problem.

\subsection{\label{subsec:Fourier-Analysis-and}Fourier Analysis and Krawtchouk
Polynomials}

A subset $A$ is uniquely determined by its characteristics function
$1_{A}$. For this Boolean function $1_{A}$, the Fourier expansion
and Fourier weights are defined as follows. Consider the Fourier basis
$\{\chi_{S}\}_{S\subseteq[1:n]}$ with $\chi_{S}(\mathbf{x}):=\prod_{i\in S}x_{i}$
for $S\subseteq[1:n]$. Then for a function $f:\{-1,1\}^{n}\to\mathbb{R}$,
define its Fourier coefficients as 
\begin{align*}
 & \hat{f}_{S}:=\mathbb{E}_{\mathbf{X}\sim\mathrm{Unif}\{-1,1\}^{n}}[f(\mathbf{X})\chi_{S}(\mathbf{X})],\;S\subseteq[1:n].
\end{align*}
Then the Fourier expansion of the function $f$ (cf. \citep[Equation (1.6)]{O'Donnell14analysisof})
is $f(\mathbf{x})=\sum_{S\subseteq[1:n]}\hat{f}_{S}\chi_{S}(\mathbf{x}).$
The \emph{degree-$k$ Fourier weight }of $f$ is defined as $\mathbf{W}_{k}[f]:=\sum_{S:|S|=k}\hat{f}_{S}^{2},\;k\in[0:n].$
For brevity, we denote $\mathbf{W}_{k}[f]$ as $\mathbf{W}_{k}$.
By definition, it is easily seen that for $f=1_{A}$, $\mathbf{W}_{0}=\alpha^{2}\textrm{ and }\sum_{k=0}^{n}\mathbf{W}_{k}=\alpha,$
where $\alpha=|A|/2^{n}$. For a code $A\subseteq\{-1,1\}^{n}$, define
the \emph{scaled degree-$k$ Fourier weight} of $1_{A}$ as $Q^{(A)}(k):=\frac{1}{\alpha^{2}}\mathbf{W}_{k}.$
If $A$ is a linear code, then $Q^{(A)}$ is the distance distribution
of the dual of code $A$, and hence is also called the \emph{dual
distribution} of $A$. For details, please refer to \cite{macwilliams1977theory}.
By definition, 
\begin{equation}
Q^{(A)}(0)=1,\ \sum_{i=0}^{n}Q^{(A)}(i)=\frac{1}{\alpha},\ \textrm{ and }Q^{(A)}(i)\ge0\textrm{ for }i\in[0:n].\label{eq:Qproperty}
\end{equation}

For each $k\in[0:n]$ and indeterminate $x$, the Krawtchouk polynomials
\cite{macwilliams1977theory} are defined as\footnote{ { Here for a real number $x$ and an integer $j$, the (generalized)
binomial coefficients ${x \choose j}:=\frac{x(x-1)\cdots(x-j+1)}{j!}$
if $j>0$; ${x \choose j}:=1$ if $j=0$; and ${x \choose j}:=0$
if $j<0$. Obviously, by definition, ${n \choose j}\ge0$ for nonnegative
integer $n$. In particular, for nonnegative integer $n$, ${n \choose j}=0$
if $j<0$ or $n<j$. See more properties on \citet[pp. 13-14]{macwilliams1977theory}. }} 
\begin{equation}
K_{k}^{(n)}(x):=\sum_{j=0}^{k}(-1)^{j}{x \choose j}{n-x \choose k-j},\label{eq:-24}
\end{equation}
whose generating function satisfies 
\begin{equation}
\sum_{k=0}^{\infty}K_{k}^{(n)}(x)z^{k}=(1-z)^{x}(1+z)^{n-x}.\label{eq:-22-1}
\end{equation}
For brevity and if there is no ambiguity, we denote $K_{k}^{(n)}$
as $K_{k}$. It is well-known that (see, e.g., \cite{macwilliams1977theory})
\begin{align}
 & K_{0}(i)=1,\;\forall i\in[0:n];\nonumber \\
 & K_{k}(0)={n \choose k},\;K_{k}(1)={n \choose k}(1-\frac{2k}{n}),\nonumber \\
 & {K_{k}(2)={n \choose k}(1-\frac{4k(n-k)}{n(n-1)})\;\forall k\in[0:n];}\nonumber \\
 & \sum_{S\subseteq[1:n]}\chi_{S}(\mathbf{x})\chi_{S}(\mathbf{x}')z^{|S|}=(1-z)^{d_{\mathrm{H}}(\mathbf{x},\mathbf{x}')}(1+z)^{n-d_{\mathrm{H}}(\mathbf{x},\mathbf{x}')},\;\forall\mathbf{x},\mathbf{x}'\in\{-1,1\}^{n};\nonumber \\
 & K_{k}(d_{\mathrm{H}}(\mathbf{x},\mathbf{x}'))=\sum_{S:|S|=k}\chi_{S}(\mathbf{x})\chi_{S}(\mathbf{x}'),\;\forall k\in[0:n],\mathbf{x},\mathbf{x}'\in\{-1,1\}^{n}.\label{eq:-19}
\end{align}
Taking expectation for both sides of \eqref{eq:-19}, with respect
to $(\mathbf{X},\mathbf{X}')\sim\mathrm{Unif}^{\otimes2}(A)$, yields
the following relationship between $P^{(A)}$ and $Q^{(A)}$: 
\begin{align}
Q^{(A)}(k) & =\sum_{i=0}^{n}P^{(A)}(i)K_{k}(i)\textrm{ and }P^{(A)}(k)=\frac{1}{2^{n}}\sum_{i=0}^{n}Q^{(A)}(i)K_{k}(i).\label{eq:-23-1}
\end{align}
These are so-called MacWilliams--Delsarte identities \cite{macwilliams1977theory}.

We now provide the following extremal property of Krawtchouk polynomials.
The proof of Lemma \ref{lem:exactextreme} is provided in  \ref{sec:Proof-of-Lemma}. 
\begin{lem}
\label{lem:exactextreme} For an integer $n\ge1$, the following hold: 
\begin{enumerate}
\item For $0\le k,i\le n$, we have $K_{k}^{(n)}(0)\ge|K_{k}^{(n)}(i)|.$ 
\item For $0\le k\le\frac{n-1}{2}$ and $1\le i\le n-1$, we have $K_{k}^{(n)}(1)\ge|K_{k}^{(n)}(i)|.$ 
\item For $0\le k\le\tau(n)$ (defined in \eqref{eq:tau}) and $2\le i\le n-2$,
we have $K_{k}^{(n)}(2)\ge|K_{k}^{(n)}(i)|.$ 
\end{enumerate}
\end{lem}

Note that the upper threshold $\tau(n)$ in Statement 3 is not sharp.
Numerical simulation shows that the upper threshold can be sharpened
to a value close to $n/2$. Proving this seems not easy. However,
it can be proven when $n$ is sufficiently large; see the following
lemma. This lemma is just a strengthening of \citep[Lemma 3]{yu2019improved}
and the proof is similar to that of \citep[Lemma 3]{yu2019improved}.
Hence we omit the proof of the following lemma. 
\begin{lem}
\label{lem:asymextreme} Given a non-negative integer $i$ and $\delta>0$,
for all sufficiently large $n$, 
\begin{equation}
K_{k}^{(n)}(i)\ge|K_{k}^{(n)}(x)|,\;\forall k\in[\delta n:(\frac{1}{2}-\delta)n],x\in[i,n-i].\label{eq:asymbound}
\end{equation}
\end{lem}

Combining Lemmas \ref{lem:exactextreme} and \ref{lem:asymextreme}
yields that given $\delta>0$, for all sufficiently large $n$, 
\begin{equation}
K_{k}^{(n)}(2)\ge|K_{k}^{(n)}(j)|,\;\forall k\in[0:(\frac{1}{2}-\delta)n],j\in[2,n-2].\label{eq:asymbound-1}
\end{equation}

\subsection{\label{subsec:Linear-Program-and}Linear Program and Its Dual}

By the MacWilliams--Delsarte identity \eqref{eq:-23-1}, 
\begin{equation}
\sum_{k=0}^{r}P^{(A)}(k)=\frac{1}{2^{n}}\sum_{i=0}^{n}Q^{(A)}(i)\sum_{k=0}^{r}K_{k}(i).\label{eq:-17}
\end{equation}
We define 
\begin{equation}
\omega_{i}^{(r)}:=\sum_{k=0}^{r}K_{k}(i)=K_{r}^{(n-1)}(i-1)\label{eq:w}
\end{equation}
(for the equality, see \citep[Equation (54)]{levenshtein1995krawtchouk}),
and for brevity, we also denote $\omega_{i}^{(r)}$ as $\omega_{i}$.
Then, 
\begin{equation}
\sum_{k=0}^{r}P^{(A)}(k)=\frac{1}{2^{n}}\sum_{i=0}^{n}Q^{(A)}(i)\omega_{i}.\label{eq:-20}
\end{equation}
Note that, in particular, $\omega_{0}={n \choose \le r}$ and $\omega_{1}={n-1 \choose r}$.
From \eqref{eq:Qproperty}, \eqref{eq:-23-1}, and $P^{(A)}(k)\ge0$,
the following properties of $Q^{(A)}$ hold: 
\begin{align}
 & Q^{(A)}(i)\ge0,\;i\in[0:n],\quad Q^{(A)}(0)=1,\quad\sum_{i=0}^{n}Q^{(A)}(i)=\frac{1}{\alpha},\label{eq:-77}\\
\textrm{ and } & \sum_{i=0}^{n}Q^{(A)}(i)K_{k}(i)\ge0,\;k\in[0:n].\label{eq:Q0-1}
\end{align}
Substituting \eqref{eq:-77} into \eqref{eq:-20}, we obtain 
\begin{align*}
\sum_{k=0}^{r}P^{(A)}(k) & =\frac{1}{2^{n}}[\omega_{0}+(\frac{1}{\alpha}-1-\sum_{i=2}^{n}Q^{(A)}(i))\omega_{1}+\sum_{i=2}^{n}Q^{(A)}(i)\omega_{i}]\\
 & =\frac{1}{2^{n}}[\omega_{0}+(\frac{1}{\alpha}-1)\omega_{1}-\sum_{i=2}^{n}Q^{(A)}(i)(\omega_{1}-\omega_{i})].
\end{align*}
We now consider a relaxed version of the minimization of $\sum_{i=2}^{n}Q^{(A)}(i)(\omega_{1}-\omega_{i})$
over the dual distance distribution $Q^{(A)}$. Instead of the discrete
optimization of $\sum_{i=2}^{n}Q^{(A)}(i)(\omega_{1}-\omega_{i})$
(since given $n$, there are only finitely many codes and the corresponding
dual distance distributions), we allow $(Q^{(A)}(0),Q^{(A)}(1),...,Q^{(A)}(n))$
to be any nonnegative vector $(u_{0},u_{1},...,u_{n})$ such that
\begin{align*}
 & u_{0}=1,u_{i}\ge0,\;i\in[2:n];\quad\sum_{i=0}^{n}u_{i}=\frac{1}{\alpha};\quad\sum_{i=0}^{n}u_{i}K_{k}(i)\ge0,\;k\in[0:n].
\end{align*}
Then in order to lower bound $\sum_{i=2}^{n}Q^{(A)}(i)(\omega_{1}-\omega_{i})$,
we consider the following linear program. 
\begin{problem}[Primal Problem]
\label{prob:Primal-Problem:}
\[
\Lambda_{n}(\alpha,r):=\min_{u_{2},u_{3},...,u_{n}}\sum_{i=2}^{n}u_{i}(\omega_{1}-\omega_{i})
\]
subject to the inequalities 
\begin{align*}
 & u_{i}\ge0,\;i\in[2:n];\\
 & \sum_{i=2}^{n}[K_{k}(1)-K_{k}(i)]u_{i}\le K_{k}(0)+K_{k}(1)(\frac{1}{\alpha}-1),\;k\in[1:n].
\end{align*}
\end{problem}

The dual is the following optimization problem. 
\begin{problem}[Dual Problem ]
\label{prob:Dual-Problem:}
\begin{equation}
\overline{\Lambda}_{n}(\alpha,r):=\max_{x_{1},x_{2},...,x_{n}}-\sum_{k=1}^{n}[K_{k}(0)+K_{k}(1)(\frac{1}{\alpha}-1)]x_{k}\label{eq:-55}
\end{equation}
subject to the inequalities 
\begin{align}
 & x_{k}\ge0,\;k\in[1:n];\nonumber \\
 & \sum_{k=1}^{n}[K_{k}(1)-K_{k}(i)]x_{k}\ge-(\omega_{1}-\omega_{i}),\;i\in[2:n].\label{eq:-122}
\end{align}
\end{problem}

By strong duality in linear programming,\footnote{Obviously, in the primal problem, since $u_{i}\ge0$, the primal problem
is bounded. On the other hand, the existence of a code $A$ with size
$M:=\alpha2^{n}$ ensures that $u_{i}=Q^{(A)}(i)$ is a feasible solution.
Hence the primal problem has an optimal solution.} $\Lambda_{n}(\alpha,r)=\overline{\Lambda}_{n}(\alpha,r)$. Therefore,
the following holds. 
\begin{thm}
\label{thm:-For-any-1} For any code $A$ of size $M$, $\sum_{i=2}^{n}Q^{(A)}(i)(\omega_{1}-\omega_{i})\geq\overline{\Lambda}_{n}(\alpha,r).$ 
\end{thm}

\subsection{\label{subsec:Linear-Programming-Bound}Linear Programming Bounds}

We next provide a lower bound for $\overline{\Lambda}_{n}(\alpha,r)$. 
\begin{thm}
\label{thm:LPB} For any code $A$ of size $M$, $\overline{\Lambda}_{n}(\alpha,r)\geq\psi_{n}^{+}(\alpha,r)$,
where $\psi_{n}^{+}$ was given in Theorem \ref{thm:LPBound}. 
\end{thm}

The proof of Theorem \ref{thm:LPB} is provided in  \ref{sec:Proof-of-Theorem}.
In our proof, we constructed different feasible solutions for different
cases. For example, for even $r\le n/2-1$ our feasible solution is
$\mathbf{x}^{*}=(0,...,0,x_{k}^{*},x_{k+1}^{*},0,...,0)$ with $x_{k}^{*}=x_{k+1}^{*}=\frac{n{n-2 \choose r-1}}{{n \choose k}(2k-n)},$
where $k$ is an odd number such that $n-\tau(n)\le k\le n$. The
feasibility of this solution follows since, on one hand, such a solution
guarantees that equality holds in \eqref{eq:-122} for $i=2,n$; and
on the other hand, after substituting $\mathbf{x}^{*}$ into \eqref{eq:-122},
it can be found that \eqref{eq:-122} holds for $i\in[2:n]$ if and
only if it holds for $i=2,n$. Hence $\mathbf{x}^{*}$ is feasible.
It is easy to see that it leads to the bound in Theorem \ref{thm:LPB}
for even $r\le n/2-1$. Other cases are proven similarly.

If we focus on sufficiently large $n$, then we can obtain a better
bound, as shown in the following theorem. The proof of Theorem \ref{thm:AsymLPB}
is almost same as the proof of Theorem \ref{thm:LPB} except that
Lemma \ref{lem:asymextreme} (more preciously, the inequality \eqref{eq:asymbound-1}),
instead of Lemma \ref{lem:exactextreme}, is applied. 
\begin{thm}
\label{thm:AsymLPB}For any code $A$ of size $M$ and any $\delta>0$,
for sufficiently large $n$, the set ``$[n-\tau(n):n]$'' in the
first two clauses of $\psi_{n}(\alpha,r)$ in \eqref{eq:psi} can
be replaced with ``$[(\frac{1}{2}+\delta)n:n]$''. In particular,
when $n\to\infty$, for fixed $\delta>0$, 
\[
\overline{\Lambda}_{n}(\alpha,r)\geq\kappa_{n}(\alpha,r):=\begin{cases}
{n-2 \choose r-1}(\varphi(\alpha)+o_{n}(1)), & \textrm{even }r\le n(\frac{1}{2}-\delta)\\
{n-2 \choose r-1}(\hat{\varphi}(\alpha,\beta)+o_{n}(1)), & \textrm{odd }r\le n(\frac{1}{2}-\delta)\\
{n-2 \choose r+1}(\varphi(\alpha)+o_{n}(1)), & \textrm{even }r\geq n/2-1\\
{n-2 \choose r}(\varphi(\alpha)+o_{n}(1)), & \textrm{odd }r\geq n/2-1
\end{cases}
\]
holds for any $\alpha:=\frac{M}{2^{n}}\in[\delta,\frac{1}{2}],\,\beta:=r/n\in(0,\frac{1}{2}-\delta]\cup[\frac{1}{2},1]$,
where all the terms $o_{n}(1)$ are independent of $M,r$ (the ones
in first two clauses depend on $\delta$). 
\end{thm}

Theorems \ref{thm:LPB} and \ref{thm:AsymLPB} implies the following
linear programming bound on $\sum_{k=0}^{r}P^{(A)}(k)$. 
\begin{thm}[Linear Programming Bound]
\label{thm:ImprovedLPB} 1. For $\alpha=\frac{M}{2^{n}}\le\frac{1}{2}$,
\begin{align}
\sum_{k=0}^{r}P^{(A)}(k) & \leq\frac{1}{2^{n}}[\omega_{0}+(\frac{1}{\alpha}-1)\omega_{1}-\psi_{n}^{+}(\alpha,r)].\label{eq:-21}
\end{align}
2. When considering the asymptotic case as $n\to\infty$, $\psi_{n}^{+}(\alpha,r)$
in \eqref{eq:-21} can be replaced by $\kappa_{n}(\alpha,r)$. 
\end{thm}

In the perspective of ball-noise stability, Theorem \ref{thm:ImprovedLPB}
can be rewritten as the bounds in Theorem \ref{thm:LPBound}. In particular,
for the asymptotic case, it follows by the fact that for fixed $\delta>0$,
$\frac{{n \choose r}}{{n \choose \le r}}\to\frac{1-2\beta}{1-\beta}$
uniformly for all $\beta:=r/n\in[\delta,\frac{1}{2}-\delta]$; see
Lemma \ref{lem:sphere-ball-ratio} in the next section. This completes
the proof of Theorem \ref{thm:LPBound}.

\section{\label{sec:Proof-of-Theorem-HCB}Proof of Theorem \ref{thm:bsiid} }

In this section, we apply a probabilistic approach to prove Theorem
\ref{thm:bsiid}. A similar approach was also used by Polyanskiy to
study the hypercontractivity phenomenon under the sphere noise \cite{polyanskiy2019hypercontractivity}.

Similar to the i.i.d.-noise stability, the sphere- or ball-noise stability
of a function can be also expressed in terms of Fourier weights of
this function. For a random noise $\mathbf{Z}\in\{-1,1\}^{n}$, let
$\mathbb{T}$ be a noise operator such that for a function $f:\{-1,1\}^{n}\to\mathbb{R}$,
$[\mathbb{T}f](\mathbf{x})=\mathbb{E}[f(\mathbf{x}\circ\mathbf{Z})]$.
Then it is easy to verify that $\mathbb{T}\chi_{S}=(2\mathbb{P}[\chi_{S}(\mathbf{Z})=1]-1)\chi_{S}.$
Hence, for any function $f:\{-1,1\}^{n}\to\mathbb{R}$, 
\begin{equation}
\mathbb{T}f=\mathbb{T}\sum_{S\subseteq[1:n]}\hat{f}_{S}\chi_{S}=\sum_{S\subseteq[1:n]}\hat{f}_{S}\mathbb{T}\chi_{S}=\sum_{S\subseteq[1:n]}(2\mathbb{P}[\chi_{S}(\mathbf{Z})=1]-1)\hat{f}_{S}\chi_{S},\label{eq:-25}
\end{equation}
which further implies that 
\[
\mathbb{E}[f(\mathbf{X})f(\mathbf{Y})]=\langle f,\mathbb{T}f\rangle=\sum_{S\subseteq[1:n]}(2\mathbb{P}[\chi_{S}(\mathbf{Z})=1]-1)\hat{f}_{S}^{2}
\]
where $\mathbf{Y}=\mathbf{X}\circ\mathbf{Z}$. In particular, if $\mathbf{Z}\sim\mathrm{Unif}(B_{r})$,
i.e., $\mathbb{T}$ corresponds to the ball-noise operator $\mathbb{B}_{r}$,
then 
\[
\mathbb{P}[\chi_{S}(\mathbf{Z})=1]=\mathbb{P}[\prod_{i\in S}Z_{i}=1]=\mathbb{P}[\prod_{i=1}^{|S|}Z_{i}=1],
\]
where the last equality follows since the distribution of $\mathbf{Z}$
is invariant under the permutation operation. Hence, for ball-noise
operator $\mathbb{B}_{r}$, 
\begin{align}
\mathbf{BStab}_{r}[f] & =\langle f,\mathbb{B}_{r}f\rangle=\sum_{k=0}^{n}(2\mathbb{P}[\prod_{i=1}^{k}Z_{i}=1]-1)\mathbf{W}_{k}.\label{eq:-99}
\end{align}
For sphere noise operator $\mathbb{S}_{r}$, $\mathbf{SStab}_{r}[f]$
admits the same expression above, but in which $\mathbf{Z}\sim\mathrm{Unif}(S_{r})$.
We next provide the exact and asymptotic expressions for the coefficients
$2\mathbb{P}[\prod_{i=1}^{k}Z_{i}=1]-1$ for sphere noise and ball
noise. 
\begin{lem}
\label{lem:propertyBallSphNoise}1. For $\mathbf{Z}\sim\mathrm{Unif}(S_{r})$,
\begin{align}
2\mathbb{P}[\prod_{i=1}^{k}Z_{i}=1]-1 & =\frac{K_{r}(k)}{{n \choose r}}.\label{eq:-100}
\end{align}
For fixed $k$ and $\delta>0$, as $n\to\infty$, 
\begin{align}
2\mathbb{P}[\prod_{i=1}^{k}Z_{i}=1]-1 & \to(1-2\beta)^{k}\textrm{ and }\max_{j\in[k:n-k]}|2\mathbb{P}[\prod_{i=1}^{j}Z_{i}=1]-1|\to(1-2\beta)^{k}\label{eq:-112}
\end{align}
uniformly for all $\beta:=r/n\in[\delta,1/2-\delta]$.\\
 2. For $\mathbf{Z}\sim\mathrm{Unif}(B_{r})$, 
\begin{align}
2\mathbb{P}[\prod_{i=1}^{k}Z_{i}=1]-1 & =\frac{K_{r}^{(n-1)}(k-1)}{{n \choose \le r}}.\label{eq:-100-2}
\end{align}
For fixed $k$ and $\delta>0$, as $n\to\infty$, 
\begin{align*}
2\mathbb{P}[\prod_{i=1}^{k}Z_{i}=1]-1 & \to(1-2\beta)^{k}\textrm{ and }\max_{j\in[k:n-k+1]}|2\mathbb{P}[\prod_{i=1}^{j}Z_{i}=1]-1|\to(1-2\beta)^{k}
\end{align*}
uniformly for all $\beta:=r/n\in[\delta,1/2-\delta]$. 
\end{lem}

\begin{proof}
We first prove Statement 1. Assume $\mathbf{Z}\sim\mathrm{Unif}(S_{r})$.
By \eqref{eq:-19} with $\mathbf{x}'\leftarrow\mathbf{1}$ and $\mathbf{x}\leftarrow\mathbf{Z}$,
we have 
\begin{equation}
\mathbb{E}[K_{k}(d_{\mathrm{H}}(\mathbf{Z},\mathbf{1}))]=\mathbb{E}[\sum_{S:|S|=k}\chi_{S}(\mathbf{Z})\chi_{S}(\mathbf{1})].\label{eq:-19-1-1}
\end{equation}
The right side of \eqref{eq:-19-1-1} is equal to $\sum_{S:|S|=k}\mathbb{E}[\chi_{S}(\mathbf{Z})]={n \choose k}(2\mathbb{P}[\prod_{i=1}^{k}Z_{i}=1]-1).$
The left side of \eqref{eq:-19-1-1} is equal to $K_{k}(r)=\frac{{n \choose k}}{{n \choose r}}K_{r}(k)$,
since ${n \choose i}K_{k}(i)={n \choose k}K_{i}(k)$ holds for any
nonnegative integers $i,k$. Hence, we obtain \eqref{eq:-100}.
\begin{lem}
\label{lem:sphere-ball-ratio}1. For fixed $k\ge i\ge0$, $\frac{{n-k \choose r-i}}{{n \choose r}}\to\beta^{i}(1-\beta)^{k-i}$
uniformly for all $\beta:=r/n\in[0,1]$.\\
 2. For fixed $\delta>0$, $\frac{{n \choose \le r}}{{n \choose r}}\to\sum_{j=0}^{r}(\frac{\beta}{1-\beta})^{j}$
uniformly for all $\beta:=r/n\in[0,\frac{1}{2}-\delta]$.\\
 3. For fixed $k\ge i\ge0$ and $\delta>0$, $\frac{{n-k \choose \leq r-i}}{{n \choose \le r}}\to\beta^{i}(1-\beta)^{k-i}$
uniformly for all $\beta:=r/n\in[\delta,\frac{1}{2}-\delta]$. 
\end{lem}

The proof of Lemma \ref{lem:sphere-ball-ratio} is provided in \ref{sec:Proof-of-Lemma-2}.

When $\mathbf{Z}\sim\mathrm{Unif}(S_{r})$, by definition, $\mathbb{P}[\prod_{i=1}^{k}Z_{i}=1]=\frac{\sum_{j=0}^{\lfloor k/2\rfloor}{k \choose 2j}{n-k \choose r-2j}}{{n \choose r}}.$
By Statement 1 of Lemma \ref{lem:sphere-ball-ratio}, for fixed $k$,
\begin{equation}
\mathbb{P}[\prod_{i=1}^{k}Z_{i}=1]\to\sum_{j=0}^{\lfloor k/2\rfloor}{k \choose 2j}\beta^{2j}(1-\beta)^{k-2j}=\frac{1+(1-2\beta)^{k}}{2}\label{eq:-9}
\end{equation}
uniformly for all $\beta:=r/n\in[0,1]$.  {The equality
above follows by the following interpretation of the middle term in
\eqref{eq:-9}. Consider $Y^{(k)}=\sum_{i=1}^{k}X_{i}$ with i.i.d.
$X_{i}\sim\mathrm{Bern}(\beta)$ (Bernoulli distribution on $\{0,1\}$
with $1$ having mass $\beta$) and denote $p_{k}:=\mathbb{P}[Y^{(k)}\textrm{ is even}]$
and $q_{k}:=\mathbb{P}[Y^{(k)}\textrm{ is odd}]$. Then, the middle
term in \eqref{eq:-9} is just $p_{k}$. Obviously, $p_{k},q_{k}$
satisfy the following recursive relations: $p_{k}=(1-\beta)p_{k-1}+\beta q_{k-1}$
and $q_{k}=(1-\beta)q_{k-1}+\beta p_{k-1}$, which imply $p_{k}-q_{k}=(1-2\beta)(p_{k-1}-q_{k-1})=...=(1-2\beta)^{k}$.
Combining this with $p_{k}+q_{k}=1$ yields the equality in \eqref{eq:-9}.
Obviously, \eqref{eq:-9} is just restatement of the first convergence
result  in \eqref{eq:-112}. That is, we complete the proof of the
first convergence result  in \eqref{eq:-112}. }

The second convergence result in \eqref{eq:-112} follows from the
first one in \eqref{eq:-112}, \eqref{eq:-100}, and Lemma \ref{lem:asymextreme}.
This completes the proof of Statement 1.

Statement 2 follows similarly, and hence the proof is omitted here. 
\end{proof}
Substituting \eqref{eq:-100-2} into \eqref{eq:-99}, we obtain 
\[
\mathbf{BStab}_{r}[f]=\frac{1}{{n \choose \le r}}\sum_{k=0}^{n}\mathbf{W}_{k}K_{r}^{(n-1)}(k-1).
\]
Similarly, 
\begin{align}
\mathbf{SStab}_{r}[f] & =\frac{1}{{n \choose r}}\sum_{k=0}^{n}\mathbf{W}_{k}K_{r}^{(n)}(k).\label{eq:-3}
\end{align}
In contrast, note that \citep[Theorem 2.49]{O'Donnell14analysisof}
for i.i.d. noise, 
\begin{align}
\mathbf{Stab}_{\beta}[f] & =\sum_{k=0}^{n}\mathbf{W}_{k}(1-2\beta)^{k}.\label{eq:-2}
\end{align}

By Lemma \ref{lem:propertyBallSphNoise}, we obtain several relationships
between $\mathbf{SStab}_{r}[f_{n}]$, $\mathbf{BStab}_{r}[f_{n}]$,
and $\mathbf{Stab}_{\beta}[f_{n}]$ for a sequence of functions $\{f_{n}\}$.
Before introducing these relationships, we first introduce a new concept
on the tail behavior of Fourier weights. 
\begin{defn}
\label{def:Fourierweightstable} A sequence of functions $\{f_{n}:\{-1,1\}^{n}\to\mathbb{R}\}_{n=1}^{\infty}$
is called \emph{Fourier-weight stable} if $\lim_{n\to\infty}\sum_{k=n-k_{0}}^{n}\mathbf{W}_{k}[f_{n}]=0, \forall  k_{0} \ge 0$ (or equivalently, $\lim_{k_{0}\to\infty}\lim_{n\to\infty}\sum_{k=n-k_{0}}^{n}\mathbf{W}_{k}[f_{n}]=0$).
\end{defn}

\begin{thm}
\label{thm:1)-Let-}1. Let $\{f_{n}\}$ be a sequence of  {nonnegative}
(not necessarily Boolean) functions with bounded $L^{2}$-norm, i.e.,
$\limsup_{n\to\infty}\mathbb{E}[f_{n}^{2}(\mathbf{X})]<\infty$. Let
$\delta>0$. Given $n$, let $r\in\mathbb{N}$ such that $\beta:=r/n\in[\delta,1/2-\delta]$.
Then, for even $r$, 
\begin{align}
\mathbf{Stab}_{\beta}[f_{n}]+o_{n}(1)\leq\mathbf{SStab}_{r}[f_{n}] & \leq2\mathbf{Stab}_{\beta}[f_{n}]+o_{n}(1),\label{eq:sphereeven}\\
\mathbf{Stab}_{\beta}[f_{n}]+o_{n}(1)\leq\mathbf{BStab}_{r}[f_{n}] & \leq2(1-\beta)\mathbf{Stab}_{\beta}[f_{n}]+o_{n}(1),\label{eq:balleven}
\end{align}
and for odd $r$, 
\begin{align}
0\leq\mathbf{SStab}_{r}[f_{n}] & \leq\mathbf{Stab}_{\beta}[f_{n}]+o_{n}(1),\label{eq:sphereodd}\\
2\beta\mathbf{Stab}_{\beta}[f_{n}]+o_{n}(1)\leq\mathbf{BStab}_{r}[f_{n}] & \leq\mathbf{Stab}_{\beta}[f_{n}]+o_{n}(1).\label{eq:ballodd}
\end{align}
Here, the terms $o_{n}(1)$ in \eqref{eq:sphereeven}-\eqref{eq:ballodd}
are independent of $r$ (given $n$), but dependent of $\delta$.
\\
2. Any sequence of  {nonnegative} functions $f_{n}$
supported on a subset of $\{\mathbf{x}:d_{\mathrm{H}}(\mathbf{x},\mathbf{1})\textrm{ is even}\}$
(or $\{\mathbf{x}:d_{\mathrm{H}}(\mathbf{x},\mathbf{1})\textrm{ is odd}\}$)
attains the upper bounds in \eqref{eq:sphereeven} and \eqref{eq:balleven}
for even $r$, and the lower bounds in \eqref{eq:sphereodd} and \eqref{eq:ballodd}
for odd $r$. \\
3. Any Fourier-weight stable sequence of  {nonnegative}
functions $f_{n}$ (e.g., Hamming subcubes or Hamming balls) attains
the lower bounds in \eqref{eq:sphereeven} and \eqref{eq:balleven}
for even $r$, and the upper bounds in \eqref{eq:sphereodd} and \eqref{eq:ballodd}
for odd $r$.
\end{thm}

\begin{proof}
We first consider the sphere noise case. By Lemma \ref{lem:propertyBallSphNoise},
for fixed $k$, 
\begin{align*}
2\mathbb{P}[\prod_{i=1}^{k}Z_{i}=1]-1 & =(1-2\beta)^{k}+o_{n}(1),\\
2\mathbb{P}[\prod_{i=1}^{n-k}Z_{i}=1]-1 & =\frac{K_{r}(n-k)}{{n \choose r}}=\frac{(-1)^{r}}{{n \choose r}}K_{r}(k)=(-1)^{r}(1-2\beta)^{k}+o_{n}(1),
\end{align*}
and for fixed $k_{0}$, 
\begin{align*}
\max_{k\in[k_{0}+1:n-k_{0}-1]}|2\mathbb{P}[\prod_{i=1}^{k}Z_{i}=1]-1| & =(1-2\beta)^{k_{0}+1}+o_{n}(1).
\end{align*}
By these equalities, we have that for fixed $k_{0}$, 
\begin{align*}
\mathbf{SStab}_{r}[f_{n}] & =\sum_{k=0}^{n}\mathbf{W}_{k}(2\mathbb{P}[\prod_{i=1}^{k}Z_{i}=1]-1)\lessgtr a(k_{0})+(-1)^{r}b(k_{0})\pm c(k_{0})+o_{n}(1).
\end{align*}
where  {
\begin{equation}
a(k_{0})=\sum_{k=0}^{k_{0}}\mathbf{W}_{k}(1-2\beta)^{k},\;b(k_{0})=\sum_{k=0}^{k_{0}}\mathbf{W}_{n-k}(1-2\beta)^{k},\;c(k_{0})=(\sum_{k=k_{0}+1}^{n-k_{0}-1}\mathbf{W}_{k})(1-2\beta)^{k_{0}+1}.\label{eq:-26}
\end{equation}
}For odd $r$, 
\begin{align}
0\le\mathbf{SStab}_{r}[f_{n}] & =a(k_{0})-b(k_{0})+c(k_{0})+o_{n}(1).\label{eq:-18}
\end{align}
Equation \eqref{eq:-18} implies \eqref{eq:sphereodd} since $a(k_{0})\leq\sum_{k=0}^{n}\mathbf{W}_{k}(1-2\beta)^{k}=\mathbf{Stab}_{\beta}[f_{n}]$
and $c(k_{0})\leq\limsup_{n\to\infty}\mathbb{E}[f_{n}^{2}](1-2\beta)^{k_{0}+1}\to0$
as $k_{0}\to\infty$.

We now consider even $r$. For this case, 
\begin{align}
\mathbf{SStab}_{r}[f_{n}] & \lessgtr a(k_{0})+b(k_{0})\pm c(k_{0})+o_{n}(1).\label{eq:-30}
\end{align}
Since $\mathbf{Stab}_{\beta}[f_{n}]=\sum_{k=0}^{n}\mathbf{W}_{k}(1-2\beta)^{k}\le a(k_{0})+b(k_{0})+c(k_{0})$
and $c(k_{0})\to0$ as $k_{0}\to\infty$. We have $\mathbf{SStab}_{r}[f_{n}]\ge\mathbf{Stab}_{\beta}[f_{n}]+o_{n}(1)$. 

On the other hand, for  even $r$, \eqref{eq:-18} also implies that
for a sequence of functions $\{f_{n}\}$, 
\begin{equation}
b(k_{0})\leq a(k_{0})+c(k_{0})+o_{n}(1),\label{eq:-10}
\end{equation}
where the $\beta$'s in definitions of $a,b,c$ in \eqref{eq:-26}
are reset to $(r-1)/n=\beta-1/n$ so that the radius $r-1$ is odd.
Note that $1/n$ vanishes as $n\to\infty$, and hence, this asymptotically
vanishing term can be merged into $o_{n}(1)$, which means that \eqref{eq:-10}
still holds if $a,b,c$ remain unchanged as in \eqref{eq:-26} (in
other words, the $\beta$'s there are $r/n$). Substituting \eqref{eq:-10}
into \eqref{eq:-30} yields that for even $r$, 
\begin{align*}
\mathbf{SStab}_{r}[f_{n}] & \leq2a(k_{0})+2c(k_{0})+o_{n}(1)\leq2\mathbf{Stab}_{\beta}[f_{n}]+2c(k_{0})+o_{n}(1).
\end{align*}
Since $k_{0}>0$ is arbitrary and $c(k_{0})\to0$ as $k_{0}\to\infty$,
we obtain \eqref{eq:sphereeven}.

For ball noise case, inequalities \eqref{eq:balleven} and \eqref{eq:ballodd}
can be proven similarly. The proof is omitted. Furthermore, if $f$
is supported on a subset of $\{\mathbf{x}:d_{\mathrm{H}}(\mathbf{x},\mathbf{1})\textrm{ is even}\}$,
then by definition, the Fourier coefficients of $f$ satisfy that
$\hat{f}_{S}=\hat{f}_{S^{c}}$ for any $S\subseteq[1:n]$. Hence,
Statement 2 can be easily verified. In addition, Statement 3 can
be easily verified as well. 
\end{proof}
We next turn back to prove Theorem \ref{thm:bsiid}. 
\begin{proof}[Proof of Theorem \ref{thm:bsiid}]
The upper bound in \eqref{eq:-116} and the upper and lower bounds
in \eqref{eq:-117}, as well as $\Gamma_{\mathrm{odd,S}}(\alpha,\beta),\,\Gamma_{\mathrm{odd,B}}(\alpha,\beta)\leq\Gamma_{\mathrm{IID}}(\alpha,\beta)$
and $\underline{\Gamma}_{\mathrm{even,S}}(\alpha,\beta)\geq\Gamma_{\mathrm{IID}}(\alpha,\beta)$
follow directly from Theorem \ref{thm:1)-Let-}. It remains to prove
$\Gamma_{\mathrm{odd,S}}(\alpha,\beta),\,\Gamma_{\mathrm{odd,B}}(\alpha,\beta)\geq\Gamma_{\mathrm{IID}}(\alpha,\beta)$
and $\underline{\Gamma}_{\mathrm{even,S}}(\alpha,\beta)\geq\frac{1}{2}\Gamma_{\mathrm{IID}}(2\alpha,\beta).$

We first prove $\Gamma_{\mathrm{odd,S}}(\alpha,\beta)\geq\Gamma_{\mathrm{IID}}(\alpha,\beta)$.
Let $A\subseteq\{-1,1\}^{n}$ be a subset of size $M$. Now we construct
a new subset $B_{k}=A\times\{-1,1\}^{k}.$ Obviously, $B_{k}\subseteq\{-1,1\}^{n+k}$
and $|B_{k}|=2^{k}M$. Next, we prove that for fixed $n,A$, 
\begin{equation}
\lim_{k\to\infty}\mathbf{SStab}_{2\lfloor\frac{(n+k)\beta}{2}\rfloor+1}[1_{B_{k}}]\geq\mathbf{Stab}_{\beta}[1_{A}].\label{eq:-118}
\end{equation}

For any $\mathbf{x}\in B_{k}$, we can write $\mathbf{x}=(\mathbf{x}_{1},\mathbf{x}_{2})$
where $\mathbf{x}_{1}\in A$ and $\mathbf{x}_{2}\in\{-1,1\}^{k}$.
Then we have 
\begin{equation}
d_{\mathrm{H}}(\mathbf{x},\mathbf{y})=d_{\mathrm{H}}(\mathbf{x}_{1},\mathbf{y}_{1})+d_{\mathrm{H}}(\mathbf{x}_{2},\mathbf{y}_{2}).\label{eq:-68-1}
\end{equation}
Denote $r_{k}:=2\lfloor\frac{(n+k)\beta}{2}\rfloor+1$. Using \eqref{eq:-68-1}
we obtain that under sphere noise, 
\begin{align}
\mathbf{SStab}_{r_{k}}[1_{B_{k}}] & =\mathbb{P}[\mathbf{X}\in B_{k},\mathbf{Y}\in B_{k}]\nonumber \\
 & =\frac{\#\{(\mathbf{x},\mathbf{y})\in B_{k}^{2}:d_{\mathrm{H}}(\mathbf{x},\mathbf{y})=r_{k}\}}{2^{n+k}{n+k \choose r_{k}}}\nonumber \\
 & =\sum_{i=0}^{n}\#\{(\mathbf{x}_{1},\mathbf{y}_{1})\in A^{2}:d_{\mathrm{H}}(\mathbf{x}_{1},\mathbf{y}_{1})=i\}\frac{\#\{(\mathbf{x}_{2},\mathbf{y}_{2}):d_{\mathrm{H}}(\mathbf{x}_{2},\mathbf{y}_{2})=r_{k}-i\}}{2^{n+k}{n+k \choose r_{k}}}\nonumber \\
 & =\sum_{i=0}^{n}\#\{(\mathbf{x}_{1},\mathbf{y}_{1})\in A^{2}:d_{\mathrm{H}}(\mathbf{x}_{1},\mathbf{y}_{1})=i\}\frac{{k \choose r_{k}-i}}{2^{n}{n+k \choose r_{k}}}\nonumber \\
 & \to\frac{1}{2^{n}}\sum_{i=0}^{n}\#\{(\mathbf{x}_{1},\mathbf{y}_{1})\in A^{2}:d_{\mathrm{H}}(\mathbf{x}_{1},\mathbf{y}_{1})=i\}\beta^{i}(1-\beta)^{n-i}\textrm{ as }k\to\infty\label{eq:-22}\\
 & =\mathbf{Stab}_{\beta}[1_{A}],\nonumber 
\end{align}
where \eqref{eq:-22} follows by Lemma \ref{lem:sphere-ball-ratio}.
Therefore, \eqref{eq:-118} holds, which implies $\Gamma_{\mathrm{odd,S}}(\alpha,\beta)\geq\Gamma_{\mathrm{IID}}(\alpha,\beta).$
Similarly, one can prove $\Gamma_{\mathrm{odd,B}}(\alpha,\beta)\geq\Gamma_{\mathrm{IID}}(\alpha,\beta).$

We next prove $\underline{\Gamma}_{\mathrm{even,S}}(\alpha,\beta)\geq\frac{1}{2}\Gamma_{\mathrm{IID}}(2\alpha,\beta)$.
We have shown that $\Gamma_{\mathrm{odd,S}}(\alpha,\beta)=\Gamma_{\mathrm{IID}}(\alpha,\beta)$
holds. We now claim that $\Gamma_{\mathrm{odd,S}}(\alpha,\beta)$
is attained by a Fourier-weight stable sequence of Boolean functions,
and moreover, this sequence also attains $\Gamma_{\mathrm{IID}}(\alpha,\beta)$.
 We now prove it. For any optimal sequence of Boolean functions $\{f_{n}\}$
attaining $\Gamma_{\mathrm{odd,S}}(\alpha,\beta)$, it holds that
$\mathbf{Stab}_{\beta}[f_{n}]\ge a(k_{0})$, where $a(k_{0})$ was
defined in \eqref{eq:-26}. Combining this inequality with \eqref{eq:-18}
yields that 
\begin{align}
\mathbf{SStab}_{r}[f_{n}] & \le\mathbf{Stab}_{\beta}[f_{n}]-b(k_{0})+c(k_{0})+o_{n}(1).\label{eq:-18-1}
\end{align}
Taking limits as $n\to\infty$, we obtain 
\[
\Gamma_{\mathrm{odd,S}}(\alpha,\beta)\le\liminf_{n\to\infty}\mathbf{Stab}_{\beta}[f_{n}]-b(k_{0})+c(k_{0})\le\Gamma_{\mathrm{IID}}(\alpha,\beta)-\limsup_{n\to\infty}b(k_{0})+(1-2\beta)^{k_{0}+1}.
\]
Then, taking limits as $k_{0}\to\infty$ and by using the equality
$\Gamma_{\mathrm{odd,S}}(\alpha,\beta)=\Gamma_{\mathrm{IID}}(\alpha,\beta)$,
we have that $\lim_{k_{0}\to\infty}\limsup_{n\to\infty}b(k_{0})=0$, 
or equivalently,  $\limsup_{n\to\infty}b(k_{0})=0,\forall k_{0}\ge 0$,
which implies that $\lim_{n\to\infty}\sum_{k=n-k_{0}}^{n}\mathbf{W}_{k}[f_{n}]=0, \forall  k_{0} \ge 0$,
i.e., $\{f_{n}\}$ is Fourier-weight stable.
Moreover, we also have that $\liminf_{n\to\infty}\mathbf{Stab}_{\beta}[f_{n}]$
is equal to $\Gamma_{\mathrm{IID}}(\alpha,\beta)$.   Hence, $\Gamma_{\mathrm{IID}}(\alpha,\beta)$
is attained by $\{f_{n}\}$ as well, i.e., the claim is true.

We also need the following decomposition of a Boolean function. Any
Boolean $f$ can be written as $f=f_{\mathrm{even}}+f_{\mathrm{odd}},$
where $f_{\mathrm{even}}=\frac{1+\chi_{[1:n]}}{2}f$ and $f_{\mathrm{odd}}=\frac{1-\chi_{[1:n]}}{2}f$
are Boolean functions respectively supported on vectors $\mathbf{x}$
of even and odd Hamming weights $d_{\mathrm{H}}(\mathbf{x},\mathbf{1})$.
In fact, if $A$ is the support of $f$, then the supports of $f_{\mathrm{even}}$
and $f_{\mathrm{odd}}$ are respectively the even part and odd part
of $A$; see Definition \ref{def:even-part}. For functions $f_{\mathrm{even}},f_{\mathrm{odd}}$,
their Fourier coefficients satisfy that 
\begin{align*}
\hat{f}_{\mathrm{even},S} & =\mathbb{E}_{\mathbf{X}\sim\mathrm{Unif}\{-1,1\}^{n}}[f(\mathbf{X})\frac{1+\chi_{[1:n]}(\mathbf{X})}{2}\chi_{S}(\mathbf{X})]\\
 & =\mathbb{E}_{\mathbf{X}\sim\mathrm{Unif}\{-1,1\}^{n}}[f(\mathbf{X})\frac{\chi_{S}(\mathbf{X})+\chi_{S^{c}}(\mathbf{X})}{2}]=\frac{\hat{f}_{S}+\hat{f}_{S^{c}}}{2}
\end{align*}
and $\hat{f}_{\mathrm{odd},S}=\frac{\hat{f}_{S}-\hat{f}_{S^{c}}}{2}.$
Define the Fourier weights $\mathbf{W}_{\mathrm{even},k}:=\sum_{S:|S|=k}\hat{f}_{\mathrm{even},S}^{2}$
and $\mathbf{W}_{\mathrm{odd},k}:=\sum_{S:|S|=k}\hat{f}_{\mathrm{odd},S}^{2}$.
 {From \eqref{eq:-25}, under sphere noise, 
\begin{align}
\mathbb{E}[f_{\mathrm{even}}(\mathbf{X})f_{\mathrm{odd}}(\mathbf{Y})] & =\langle f_{\mathrm{even}},\mathbb{S}_{r}f_{\mathrm{odd}}\rangle=\sum_{S\subseteq[1:n]}\hat{f}_{\mathrm{even},S}\hat{f}_{\mathrm{odd},S}\frac{K_{r}(|S|)}{{n \choose r}}\nonumber \\
 & =\sum_{S:|S|\le n/2}\hat{f}_{\mathrm{even},S}\hat{f}_{\mathrm{odd},S}\frac{K_{r}(|S|)}{{n \choose r}}-\sum_{S:|S|\le n/2}\hat{f}_{\mathrm{even},S}\hat{f}_{\mathrm{odd},S}\frac{K_{r}(|S|)}{{n \choose r}}=0,\label{eq:-31}
\end{align}
where $\mathbf{Y}=\mathbf{X}\circ\mathbf{Z}$ with $\mathbf{Z}\sim\mathrm{Unif}(S_{r})$,
and in \eqref{eq:-31} $K_{r}(|S^{c}|)=K_{r}(|S|)$ for even $r$
is applied. }

Given $(\alpha,\beta)$, denote $\{f_{n}\}$ as an optimal Fourier-weight
stable sequence of Boolean functions with support size $\lfloor2\alpha2^{n}\rfloor$
that attains $\Gamma_{\mathrm{IID}}(2\alpha,\beta)$. For brevity,
we omit the subscript $n$ of $f_{n}$. Then for $r=2\lfloor\frac{\beta n}{2}\rfloor$,
\begin{align}
\Gamma_{\mathrm{IID}}(2\alpha,\beta)=\mathbf{Stab}_{\beta}[f] & =\mathbf{SStab}_{r}[f]+o_{n}(1)\label{eq:-27}\\
 & =\mathbb{E}[[f_{\mathrm{even}}+f_{\mathrm{odd}}](\mathbf{X})[f_{\mathrm{even}}+f_{\mathrm{odd}}](\mathbf{Y})]+o_{n}(1)\nonumber \\
 & =\mathbb{E}[f_{\mathrm{even}}(\mathbf{X})f_{\mathrm{even}}(\mathbf{Y})]+\mathbb{E}[f_{\mathrm{odd}}(\mathbf{X})f_{\mathrm{odd}}(\mathbf{Y})]+o_{n}(1)\nonumber \\
 & \leq\underline{\Gamma}_{\mathrm{even,S}}(\alpha_{\mathrm{even}},\beta)+\underline{\Gamma}_{\mathrm{even,S}}(\alpha_{\mathrm{odd}},\beta)+o_{n}(1),\label{eq:-4}
\end{align}
where \eqref{eq:-27} follows by Statement 3  of Theorem \ref{thm:1)-Let-},
and $\alpha_{\mathrm{even}}=\mathbb{E}[f_{\mathrm{even}}(\mathbf{X})]$,
$\alpha_{\mathrm{odd}}=\mathbb{E}[f_{\mathrm{odd}}(\mathbf{X})]$.
On the other hand, since $f$ is Fourier-weight stable, we have 
\begin{align*}
\mathbf{W}_{\mathrm{even},0} & =(\frac{\hat{f}_{\emptyset}+\hat{f}_{[1:n]}}{2})^{2}\leq(\frac{\sqrt{\mathbf{W}_{0}}+\sqrt{\mathbf{W}_{n}}}{2})^{2}\to\frac{\mathbf{W}_{0}}{4},\textrm{ as }n\to\infty,
\end{align*}
and similarly $\mathbf{W}_{\mathrm{odd},0}\to\frac{\mathbf{W}_{0}}{4}$
as $n\to\infty$. This means 
\begin{align}
\alpha_{\mathrm{even}} & \to\alpha\textrm{ and }\alpha_{\mathrm{odd}}\to\alpha.\label{eq:-29}
\end{align}
 {Furthermore, we claim that $\underline{\Gamma}_{\mathrm{even,S}}(\alpha,\beta)$
is continuous in $\alpha$. Let $0\le\alpha_{1}<\alpha_{2}\le1$.
We now prove this claim. Let $1_{A_{n}}$ be an optimal Boolean function
attaining $\Gamma_{\mathrm{S}}^{(n)}(\lfloor\alpha_{2}2^{n}\rfloor,2\lfloor\frac{\beta n}{2}\rfloor)$,
where $A_{n}$ is the support of this optimal Boolean function. Let
$B_{n}$ be an arbitrary subset of $A_{n}$ such that $\mathbb{P}[\mathbf{X}\in B_{n}]=\alpha_{1}$.
Then, 
\begin{align*}
\Gamma_{\mathrm{S}}^{(n)}(\lfloor\alpha_{2}2^{n}\rfloor,2\lfloor\frac{\beta n}{2}\rfloor) & =\mathbb{P}[\mathbf{X}\in A_{n},\mathbf{Y}\in A_{n}]\\
 & =\mathbb{P}[\mathbf{X}\in B_{n},\mathbf{Y}\in B_{n}]+\mathbb{P}[\mathbf{X}\in A_{n}\backslash B_{n},\mathbf{Y}\in B_{n}]\\
 & \qquad+\mathbb{P}[\mathbf{X}\in B_{n},\mathbf{Y}\in A_{n}\backslash B_{n}]+\mathbb{P}[\mathbf{X}\in A_{n}\backslash B_{n},\mathbf{Y}\in A_{n}\backslash B_{n}]\\
 & \le\mathbb{P}[\mathbf{X}\in B_{n},\mathbf{Y}\in B_{n}]+3\mathbb{P}[\mathbf{X}\in A_{n}\backslash B_{n}]\\
 & \le\Gamma_{\mathrm{S}}^{(n)}(\lfloor\alpha_{1}2^{n}\rfloor,2\lfloor\frac{\beta n}{2}\rfloor)+3\frac{\lfloor\alpha_{2}2^{n}\rfloor-\lfloor\alpha_{1}2^{n}\rfloor}{2^{n}}.
\end{align*}
Taking $\liminf_{n\to\infty}$, we obtain 
\[
\underline{\Gamma}_{\mathrm{even,S}}(\alpha_{1},\beta)\le\underline{\Gamma}_{\mathrm{even,S}}(\alpha_{2},\beta)\le\underline{\Gamma}_{\mathrm{even,S}}(\alpha_{1},\beta)+3(\alpha_{2}-\alpha_{1}),
\]
which implies the continuity of $\underline{\Gamma}_{\mathrm{even,S}}(\alpha,\beta)$
in $\alpha$. (In fact, by the same argument, other quantities such
as $\Gamma_{\mathrm{IID}}(\alpha,\beta),\overline{\Gamma}_{\mathrm{even,S}}(\alpha,\beta),\overline{\Gamma}_{\mathrm{even,B}}(\alpha,\beta)$
etc. are also continuous in $\alpha$.)}

 Finally, combining \eqref{eq:-4} and \eqref{eq:-29} and applying
the continuity of $\underline{\Gamma}_{\mathrm{even,S}}(\alpha,\beta)$
yields $\underline{\Gamma}_{\mathrm{even,S}}(\alpha,\beta)\geq\frac{1}{2}\Gamma_{\mathrm{IID}}(2\alpha,\beta).$ 

\end{proof}

\appendix
%dummy comment inserted by tex2lyx to ensure that this paragraph is not empty%dummy comment inserted by tex2lyx to ensure that this paragraph is not empty

\section{\label{sec:Proof-of-Lemma}Proof of Lemma \ref{lem:exactextreme}}

Here we prove Lemma \ref{lem:exactextreme} by using the generating
function method.

Statement 1: By the equality $\sum_{k}K_{k}^{(n)}(i)z^{k}=(1-z)^{i}(1+z)^{n-i}$
where $\sum_{k}$ means the summation over all integers (in fact,
it can be replaced by $\sum_{k=0}^{n}$ for this equality), we have
that 
\begin{align}
\sum_{k}[K_{k}^{(n)}(0)-K_{k}^{(n)}(i)]z^{k} & =(1+z)^{n-i}[(1+z)^{i}-(1-z)^{i}]\label{eq:-88}\\
 & =2(1+z)^{n-i}[\sum_{\textrm{odd }j}{i \choose j}z^{j}].\label{eq:-87}
\end{align}
Here in fact, the variable $j$ under the summation in \eqref{eq:-87}
can be additionally restricted to belong to $[0:i]$. Similarly, we
have 
\begin{align}
\sum_{k}[K_{k}^{(n)}(0)+K_{k}^{(n)}(i)]z^{k} & =2(1+z)^{n-i}[\sum_{\textrm{even }j}{i \choose j}z^{j}].\label{eq:-87-1}
\end{align}
Since all coefficients in \eqref{eq:-87} and \eqref{eq:-87-1} are
nonnegative, we have $K_{k}^{(n)}(0)\geq|K_{k}^{(n)}(i)|.$

 {Statement 2: We first prove Statement 2 for odd $i$,
i.e., the following claim. }
\begin{claim}
 {\label{claim:statement2oddi} $K_{k}^{(n)}(1)\ge|K_{k}^{(n)}(i)|$
holds for $0\le k\le\frac{n-1}{2}$ and } {\emph{odd}} {{}
$i$ such that $1\le i\le n-1$. }
\end{claim}

 {Similarly to \eqref{eq:-88}-\eqref{eq:-87}, we
obtain for $1\le i\le n-1$, 
\begin{align*}
\sum_{k}[K_{k}^{(n)}(1)-K_{k}^{(n)}(i)]z^{k} & =2(1-z)[\sum_{j}{n-i \choose j}z^{j}][\sum_{\textrm{odd }j}{i-1 \choose j}z^{j}]\\
 & =2(1-z)\sum_{k}[\sum_{\textrm{odd }j}{n-i \choose k-j}{i-1 \choose j}]z^{k}\\
 & =2\sum_{k}a_{i,k}z^{k},
\end{align*}
where 
\begin{align*}
a_{i,k} & :=\sum_{\textrm{odd }j}{i-1 \choose j}[{n-i \choose k-j}-{n-i \choose k-j-1}].
\end{align*}
By the formula ${m \choose l}={m-1 \choose l}+{m-1 \choose l-1}$,
we can rewrite $a_{i,k}=\sum_{\textrm{odd }j}{i-1 \choose j}(b_{k-j}-b_{k-j-2})$,
where $b_{l}:={n-i-1 \choose l}$. We next prove $a_{i,k}\ge0$ for
$k\le\frac{n}{2}$ and odd $i\in[1:n-1]$. }

 {   Observe that for odd $i$,
\begin{align*}
a_{i,k} & =\sum_{\textrm{odd }j}[{i-1 \choose j}-{i-1 \choose j-2}]b_{k-j}\\
 & =\sum_{\textrm{odd }j<\frac{i+1}{2}}[{i-1 \choose j}-{i-1 \choose j-2}]b_{k-j}+[{i-1 \choose i+1-j}-{i-1 \choose i-1-j}]b_{k-(i+1-j)}\\
 & =\sum_{\textrm{odd }j<\frac{i+1}{2}}[{i-1 \choose j}-{i-1 \choose j-2}][b_{k-j}-b_{k-(i+1-j)}].
\end{align*}
}

 {We now require some basic properties of binomial
coefficients. For a nonnegative integer $m$, the function $g:j\in\mathbb{Z}\mapsto{m \choose j}$
satisfies following properties. }
\begin{enumerate}
\item  {$g$ is symmetric with respect to $\frac{m}{2}$,
i.e., $g(j)=g(m-j)$.}
\item  {$g$ is nondecreasing for $j\le\frac{m+1}{2}$ and
nonincreasing for $j\ge\frac{m+1}{2}$. }
\item  {$g(j_{1})\le g(j_{2})$ for all integers $j_{1},j_{2}$
such that $j_{1}\le j_{2}$ and $j_{1}+j_{2}\le m$. }
\end{enumerate}
 {The first two properties follow by definition, and
the third one follows by the first two since $g(j_{1})\le g(j_{2})$
if $j_{2}\le\frac{m+1}{2}$, and $g(j_{1})\le g(m-j_{2})=g(j_{2})$
if $j_{2}>\frac{m+1}{2}$.}

 {By the second property above, we have ${i-1 \choose j}\ge{i-1 \choose j-2}$
since $j<\frac{i+1}{2}$ (or equivalently, $j\le\frac{i}{2}$). By
the last one, we have $b_{k-(i+1-j)}\le b_{k-j}$, since $k-(i+1-j)\le k-j$
and $k-j+k-(i+1-j)=2k-i-1\le n-i-1$.  Hence $a_{i,k}\ge0$, which
implies $K_{k}^{(n)}(1)\geq K_{k}^{(n)}(i)$ for odd $i\in[1:n-1]$.
Similarly, one can show that $K_{k}^{(n)}(1)\ge-K_{k}^{(n)}(i)$ for
odd $i\in[1:n-1]$, just by replacing all the summations above over
odd $j$ with the corresponding ones over even $j$. Hence, Claim
\ref{claim:statement2oddi} holds. }

 {We next use Claim \ref{claim:statement2oddi} to
prove Statement 2 for odd $n$ and all $i\in[1:n-1]$. From Claim
\ref{claim:statement2oddi}, $K_{k}^{(n)}(1)\ge|K_{k}^{(n)}(i)|$
for odd $i\in[1:n-1]$. On the other hand, $K_{k}^{(n)}(i)=(-1)^{k}K_{k}^{(n)}(n-i),0\le i,k\le n$
\citet[(11)]{krasikov2001nonnegative}. Hence, $K_{k}^{(n)}(1)\ge|K_{k}^{(n)}(n-i)|$
for odd $i\in[1:n-1]$. If, additionally, $n$ is odd, then $n-i$
takes all even numbers in $[1:n-1]$. Hence, for odd $n$, $K_{k}^{(n)}(1)\ge|K_{k}^{(n)}(i)|$
holds for all $i\in[1:n-1]$. }

 {We lastly prove Statement 2 for even $n$ and all
$i\in[1:n-1]$. By the conclusion for odd $n$ above, $K_{k}^{(n-1)}(1)\ge|K_{k}^{(n-1)}(i)|$
holds for $0\le k\le\frac{n-2}{2}$ and $1\le i\le n-2$, and $K_{k-1}^{(n-1)}(1)\ge|K_{k-1}^{(n-1)}(i)|$
holds for $1\le k\le\frac{n}{2}$ and $1\le i\le n-2$. By the property
that $K_{k}^{(n)}(i)=K_{k}^{(n-1)}(i)+K_{k-1}^{(n-1)}(i),1\le k\le n,0\le i\le n$
\citet[(47)]{levenshtein1995krawtchouk}, we have that for $1\le k\le\frac{n-2}{2}$
and $1\le i\le n-2$,
\begin{align*}
K_{k}^{(n)}(1) & =K_{k}^{(n-1)}(1)+K_{k-1}^{(n-1)}(1)\ge|K_{k}^{(n-1)}(i)|+|K_{k-1}^{(n-1)}(i)|\\
 & \ge|K_{k}^{(n-1)}(i)+K_{k-1}^{(n-1)}(i)|=|K_{k}^{(n)}(i)|.
\end{align*}
Hence, it remains to verify the case that $k=0,\frac{n-1}{2}$, or
$i=n-1$. First, note that $\frac{n-1}{2}$ is not an integer, and
hence, $k$ cannot equal it. We next verify the case that $k=0$ or
$i=n-1$. By definition, for $k=0$, $K_{0}^{(n)}(i)=1$, and hence,
$K_{0}^{(n)}(1)\ge|K_{0}^{(n)}(i)|$ holds obviously. For $i=n-1$,
by definition, $K_{k}^{(n)}(1)={n \choose k}(1-\frac{2k}{n})\ge0$
for $k\le n/2$, and $K_{k}^{(n)}(n-1)=(-1)^{k}K_{k}^{(n)}(1)$. Obviously,
$K_{k}^{(n)}(1)=|K_{k}^{(n)}(n-1)|$ holds. Hence, Statement 2 holds
for even $n$. Combining two points above (the cases of odd $n$ and
even $n$), Statement 2 holds for all $n\ge1$. }

Statement 3: For $2\le i\le n-2$, 
\begin{align}
\sum_{k=0}^{n}[K_{k}^{(n)}(2)-K_{k}^{(n)}(i)]z^{k} & =(1-z)^{2}(1+z)^{n-i}[(1+z)^{i-2}-(1-z)^{i-2}].\label{eq:-94-2}
\end{align}
Observe that $(1+z)^{i-2}-(1-z)^{i-2}=\sum_{\textrm{odd }j\in[0:i-2]}{i-2 \choose j}z^{j}$
and 
\begin{align}
(1-z)^{2}(1+z)^{n-i} & =(1-z)^{2}[\sum_{j=0}^{n-i}{n-i \choose j}z^{j}]\nonumber \\
 & =\sum_{j=0}^{n-i}[{n-i \choose j}+{n-i \choose j-2}-2{n-i \choose j-1}]z^{j}.\label{eq:-8}
\end{align}
It is easy to verify that ${n-i \choose j}+{n-i \choose j-2}-2{n-i \choose j-1}\ge0$
if $j\le\frac{n-i+2-\sqrt{n-i+2}}{2}$. It means that $K_{k}^{(n)}(2)\geq K_{k}^{(n)}(i)$
when $k\le\frac{n-i+2-\sqrt{n-i+2}}{2}$, since the terms $z^{j}$
in \eqref{eq:-8} with $j>k$ have no contribution to the term $z^{k}$
in the final expansion in \eqref{eq:-94-2}.

On the other hand, 
\begin{align*}
(1-z)^{2}(1+z)^{2}[(1+z)^{i-2}-(1-z)^{i-2}] & =2(1-z^{2})^{2}[\sum_{\textrm{odd }j\in[0:i-2]}{i-2 \choose j}z^{j}]\\
 & =2\sum_{\textrm{odd }j\in[0:i-2]}[{i-2 \choose j}+{i-2 \choose j-4}-2{i-2 \choose j-2}]z^{j}.
\end{align*}
It is easy to verify that ${i-2 \choose j}+{i-2 \choose j-4}-2{i-2 \choose j-2}\ge0$
if $j\le\frac{i+2-\sqrt{i+2}}{2}$. Hence $K_{k}^{(n)}(2)\geq K_{k}^{(n)}(i)$
also holds when $k\le\frac{i+2-\sqrt{i+2}}{2}$.

Combining the two points above, we have that for $2\le i\le n-2$,
$K_{k}^{(n)}(2)\geq K_{k}^{(n)}(i)$ whenever 
\[
k\le\max\{\frac{n-i+2-\sqrt{n-i+2}}{2},\frac{i+2-\sqrt{i+2}}{2}\}.
\]
Taking minimization over $2\le i\le n-2$ to find the worst case,
we have that $K_{k}^{(n)}(2)\geq K_{k}^{(n)}(i)$ holds for $2\le i\le n-2$,
when 
\begin{equation}
k\le\min_{2\le i\le n-2}\max\{\frac{n-i+2-\sqrt{n-i+2}}{2},\frac{i+2-\sqrt{i+2}}{2}\}=\frac{1}{2}(\frac{n}{2}+2-\sqrt{\frac{n}{2}+2}).\label{eq:-23}
\end{equation}

Similarly, one can show $K_{k}^{(n)}(2)\geq-K_{k}^{(n)}(i)$ if $k$
satisfies \eqref{eq:-23}.

\section{\label{sec:Proof-of-Theorem}Proof of Theorem \ref{thm:LPB}}

 {One can easily observe that $\mathbf{x}=\mathbf{0}$
is a feasible solution to Problem \ref{prob:Dual-Problem:}, since
by Lemma \ref{lem:exactextreme}, $K_{r}^{(n-1)}(0)\ge K_{r}^{(n-1)}(i-1),\forall i\in[2:n],\forall r\in[0:n-1]$.
This solution leads to the lower bound $\overline{\Lambda}_{n}(\alpha,r)\geq0.$
In the following, we construct another two kinds of feasible solutions
for different cases: the } {\emph{$1$-sparse solution}} {{}
which contains only one non-zero component, and the } {\emph{$2$-sparse
solution}} {{} which contains two non-zero components.
By using these feasible solutions, we will show that $\overline{\Lambda}_{n}(\alpha,r)\geq\psi_{n}(\alpha,r)$.
We partition all the possible cases into four parts, according to
whether $r\le n/2-1$ and whether $r$ is even.}
\begin{itemize}
\item Even $r\le n/2-1$ 
\end{itemize}
For this case, we first construct a $2$-sparse feasible solution.
Let $k$ be an odd number such that $n-\tau(n)\le k\le n-1$. Consider
the vector $\mathbf{x}^{*}:=(0,...,0,x_{k}^{*},x_{k+1}^{*},0,...,0)$
with the $k$-th and $(k+1)$-th components $(x_{k}^{*},x_{k+1}^{*})$
satisfying 
\begin{align}
 & [K_{k}(2)-K_{k}(1)]x_{k}^{*}+[K_{k+1}(2)-K_{k+1}(1)]x_{k+1}^{*}+K_{r}^{(n-1)}(1)-K_{r}^{(n-1)}(0)=0\label{eq:-80-3}\\
 & [K_{k}(n)-K_{k}(1)]x_{k}^{*}+[K_{k+1}(n)-K_{k+1}(1)]x_{k+1}^{*}+K_{r}^{(n-1)}(n-1)-K_{r}^{(n-1)}(0)=0.\label{eq:-81-2}
\end{align}
That is, if we define 
\begin{align*}
\varphi(i) & :=[K_{k}(i)-K_{k}(1)]x_{k}^{*}+[K_{k+1}(i)-K_{k+1}(1)]x_{k+1}^{*}+K_{r}^{(n-1)}(i-1)-K_{r}^{(n-1)}(0),
\end{align*}
then $\varphi(2)=\varphi(n)=0.$ Solving the equations \eqref{eq:-80-3}
and \eqref{eq:-81-2}, we obtain 
\begin{align}
x_{k}^{*} & =x_{k+1}^{*}=\frac{n{n-2 \choose r-1}}{{n \choose k}(2k-n)}.\label{eq:-59}
\end{align}
We next prove that $\mathbf{x}^{*}$ is a feasible solution to Problem
\ref{prob:Dual-Problem:}. That is, for all $i\in[2:n]$, $\varphi(i)\le0.$
By the choice of $\mathbf{x}^{*}$, we have $\varphi(2)=\varphi(n)=0$.
Hence we only need to show $\varphi(i)\le0$ for all $i\in[3:n-1]$.
We next prove this.

By the property $K_{k}(i)=(-1)^{i}K_{n-k}(i),0\le i,k\le n$ \citep[(13)]{krasikov2001nonnegative},
we have 
\begin{align*}
\varphi(i) & =[(-1)^{i}K_{n-k}(i)+K_{n-k}(1)]x_{k}^{*}+[(-1)^{i}K_{n-k-1}(i)+K_{n-k-1}(1)]x_{k+1}^{*}\\
 & \qquad+K_{r}^{(n-1)}(i-1)-K_{r}^{(n-1)}(0).
\end{align*}
By Lemma \ref{lem:exactextreme}, for $i\in[3:n-2]$ and $n-\tau(n)\le k\le n-1$,
\begin{align*}
\varphi(i) & \leq[|K_{n-k}(i)|+K_{n-k}(1)]x_{k}^{*}+[|K_{n-k-1}(i)|+K_{n-k-1}(1)]x_{k+1}^{*}+K_{r}^{(n-1)}(i-1)-K_{r}^{(n-1)}(0)\\
 & \leq[K_{n-k}(2)+K_{n-k}(1)]x_{k}^{*}+[K_{n-k-1}(2)+K_{n-k-1}(1)]x_{k+1}^{*}+K_{r}^{(n-1)}(1)-K_{r}^{(n-1)}(0)\\
 & =\varphi(2)=0.
\end{align*}
Hence, it remains to verify that $\varphi(n-1)\leq0$.

By the property $K_{k}(i)=(-1)^{k}K_{k}(n-i),0\le i,k\le n$ \citep[(11)]{krasikov2001nonnegative},
we have 
\begin{align*}
\varphi(n-1) & =[(-1)^{n-1}K_{n-k}(n-1)+K_{n-k}(1)]x_{k}^{*}+[(-1)^{n-1}K_{n-k-1}(n-1)+K_{n-k-1}(1)]x_{k+1}^{*}\\
 & \qquad+K_{r}^{(n-1)}(n-2)-K_{r}^{(n-1)}(0)\\
 & =[(-1)^{n-1+n-k}K_{n-k}(1)+K_{n-k}(1)]x_{k}^{*}+[(-1)^{n-1+n-k-1}K_{n-k-1}(1)+K_{n-k-1}(1)]x_{k+1}^{*}\\
 & \qquad+K_{r}^{(n-1)}(1)-K_{r}^{(n-1)}(0)\\
 & =2K_{n-k}(1)x_{k}^{*}+K_{r}^{(n-1)}(1)-K_{r}^{(n-1)}(0)\\
 & =2(\frac{2k}{n}-1)\frac{n{n-2 \choose r-1}}{(2k-n)}+{n-1 \choose r}(1-\frac{2r}{n-1})-{n-1 \choose r}=0.
\end{align*}

Until now, we have shown that $\mathbf{x}^{*}$ is a feasible solution
to Problem \ref{prob:Dual-Problem:}. This immediately yields the
following bound on Problem \ref{prob:Primal-Problem:}: 
\begin{align}
\overline{\Lambda}_{n}(\alpha,r) & \geq-{n \choose k}[1+(1-\frac{2k}{n})(\frac{1}{\alpha}-1)]x_{k}^{*}-{n \choose k+1}[1+(1-\frac{2(k+1)}{n})(\frac{1}{\alpha}-1)]x_{k+1}^{*}\nonumber \\
 & =\frac{n{n-2 \choose r-1}}{2k-n}[2(\frac{1}{\alpha}-1)-\frac{1}{\alpha}\frac{n+1}{k+1}].\label{eq:-6}
\end{align}

We next construct a $1$-sparse feasible solution for odd $n$. Let
$k=n$ and $\mathbf{x}^{*}:=(0,...,0,x_{n}^{*})$, where $x_{n}^{*}={n-2 \choose r-1}$
which coincides with \eqref{eq:-59} with $k=n$. For this case, all
the derivations above for the $2$-sparse feasible solution still
hold since $K_{n+1}(i)=0$ for all $0\le i\le n$.

In conclusion, for the lower bound in \eqref{eq:-6}, $k$ is allowed
to be chosen as an odd number in $[n-\tau(n):n]$. We maximize this
lower bound over all such $k$'s and obtain the desired bound. 
\begin{itemize}
\item  {Odd $r\le n/2-1$ }
\end{itemize}
 {Let $k\in[n-\tau(n):n]\cap F$ be an integer where
$F$ is defined in \eqref{eq:F}. Consider the vector $\mathbf{x}^{*}:=(0,...,0,x_{k}^{*},0,...,0)$
with the $k$-th component $x_{k}^{*}$ satisfying 
\begin{align*}
 & [K_{k}(2)-K_{k}(1)]x_{k}^{*}+K_{r}^{(n-1)}(1)-K_{r}^{(n-1)}(0)=0.
\end{align*}
That is, $x_{k}^{*}=\frac{n(n-1){n-2 \choose r-1}}{{n \choose k}k(2k-n-1)}$.
For this case, we re-define 
\begin{align*}
\varphi(i) & :=[K_{k}(i)-K_{k}(1)]x_{k}^{*}+K_{r}^{(n-1)}(i-1)-K_{r}^{(n-1)}(0)\\
 & =[(-1)^{i}K_{n-k}(i)+K_{n-k}(1)]x_{k}^{*}+K_{r}^{(n-1)}(i-1)-K_{r}^{(n-1)}(0),
\end{align*}
which satisfies $\varphi(2)=0.$ We next show $\varphi(i)\le0$ for
all $i\in[3:n]$.}

 {Similarly to the case of even $r\le n/2-1$, for
$i\in[3:n-2]$ and $n-\tau(n)\le k\le n$, one can easily verify that
$\varphi(i)\le\varphi(2)=0$. We next verify that $\varphi(n-1),\varphi(n)\leq0$.
\begin{align*}
\varphi(n-1) & =[(-1)^{n-1}K_{n-k}(n-1)+K_{n-k}(1)]x_{k}^{*}+K_{r}^{(n-1)}(n-2)-K_{r}^{(n-1)}(0)\\
 & =[(-1)^{k+1}K_{n-k}(1)+K_{n-k}(1)]x_{k}^{*}-K_{r}^{(n-1)}(1)-K_{r}^{(n-1)}(0).
\end{align*}
If $k$ is even, then obviously, $\varphi(n-1)\le0$. Otherwise, 
\begin{align*}
\varphi(n-1) & =2K_{n-k}(1)x_{k}^{*}-K_{r}^{(n-1)}(1)-K_{r}^{(n-1)}(0)\\
 & =2{n-2 \choose r-1}[(\frac{2k}{n}-1)\frac{n(n-1)}{k(2k-n-1)}-\frac{n-1}{r}+1]
\end{align*}
which is non-positive if $k>\frac{n+1}{2}$ and 
\begin{equation}
k\ge\frac{2(n-1)+s(n+1)+\sqrt{(2(n-1)-s(n+1))^{2}+8s(n-1)}}{4s}\label{eq:-7}
\end{equation}
with $s=\frac{n-1}{r}-1$. By the inequality $\sqrt{a^{2}+b^{2}}\le a+b$
for $a,b\ge0$, \eqref{eq:-7} is satisfied if 
\begin{align*}
k & \ge\max\{\frac{n+1}{2},\frac{n-1}{s}\}+\sqrt{\frac{n+1}{2s}}=\max\{\frac{n+1}{2},\frac{(n-1)r}{n-1-r}\}+\sqrt{\frac{(n+1)r}{2(n-1-r)}}.
\end{align*}
}

 {Similarly to the case of even $r\le n/2-1$, for
$\varphi(n)$, we have 
\begin{align*}
\varphi(n) & =[(-1)^{n}K_{n-k}(n)+K_{n-k}(1)]x_{k}^{*}+K_{r}^{(n-1)}(n-1)-K_{r}^{(n-1)}(0)\\
 & =[(-1)^{k}K_{n-k}(0)+K_{n-k}(1)]x_{k}^{*}-2K_{r}^{(n-1)}(0).
\end{align*}
If $k$ is odd, then obviously, $\varphi(n)\le0$. Otherwise, 
\begin{align*}
\varphi(n) & =[K_{n-k}(0)+K_{n-k}(1)]x_{k}^{*}-2K_{r}^{(n-1)}(0)\\
 & =2(n-1){n-2 \choose r-1}(\frac{1}{2k-n-1}-\frac{1}{r})
\end{align*}
which is non-positive if $k\ge\frac{n+r+1}{2}$.}

 {Therefore, the solution constructed above is feasible
if $k\in[n-\tau(n):n]\cap F$. This immediately yields the following
bound on Problem \ref{prob:Primal-Problem:}: 
\[
\overline{\Lambda}_{n}(\alpha,r)\geq\frac{(n-1){n-2 \choose r-1}}{2k-n-1}[2(\frac{1}{\alpha}-1)-\frac{1}{\alpha}\frac{n}{k}].
\]
Maximizing this lower bound over all $k\in[n-\tau(n):n]\cap F$ yields
the desired lower bound.}
\begin{itemize}
\item Odd $r>n/2-1$ 
\end{itemize}
For odd $r>n/2-1$, consider the vector $\mathbf{x}^{*}:=(0,...,0,x_{k}^{*},0,...,0)$
with the $k$-th component $x_{k}^{*}$ satisfying 
\begin{align*}
 & 2K_{n-k}(1)x_{k}^{*}+K_{n-1-r}^{(n-1)}(1)-K_{n-1-r}^{(n-1)}(0)=0,
\end{align*}
i.e., $x_{k}^{*}=\frac{n(n-1){n-2 \choose r-1}}{{n \choose k}k(2k-n-1)}$,
where $k\ge\frac{n}{2}+1$ is odd. Re-define 
\begin{align*}
\varphi(i) & :=[K_{k}(i)-K_{k}(1)]x_{k}^{*}+K_{r}^{(n-1)}(i-1)-K_{r}^{(n-1)}(0)\\
 & =[(-1)^{i+n-k}K_{n-k}(n-i)+K_{n-k}(1)]x_{k}^{*}+(-1)^{i+n-r}K_{n-1-r}^{(n-1)}(n-i)-K_{n-1-r}^{(n-1)}(0),
\end{align*}
which satisfies $\varphi(n-1)=0$. Then for odd $k$, it holds that
for all $i\in[2:n]$, 
\begin{align*}
\varphi(i) & \leq[K_{n-k}(1)+K_{n-k}(1)]x_{k}^{*}+K_{n-1-r}^{(n-1)}(1)-K_{n-1-r}^{(n-1)}(0)\\
 & =\varphi(n-1)=0.
\end{align*}
Hence, the vector $\mathbf{x}^{*}$ is feasible and leads to the desired
bound for this case.
\begin{itemize}
\item Even $r>n/2-1$ 
\end{itemize}
For even $r>n/2-1$, by Lemma \ref{lem:exactextreme}, it holds that
$K_{r+1}(1)\ge|K_{r+1}(i)|$ for all $i\in[2:n]$. This inequality
implies that $-(\omega_{1}^{(r)}-\omega_{i}^{(r)})\le-(\omega_{1}^{(r+1)}-\omega_{i}^{(r+1)})$
where $\omega_{i}^{(r)}$ is defined in \eqref{eq:w}. Hence, inequality
$\overline{\Lambda}_{n}(\alpha,r)\geq\overline{\Lambda}_{n}(\alpha,r+1)$
holds, which, combined with the lower bound for the case ``odd $r>n/2-1$'',
implies the desired lower bound for even $r>n/2-1$.

\section{\label{sec:Proof-of-Lemma-2}Proof of Lemma \ref{lem:sphere-ball-ratio}}

Statement 1 is obvious. Statement 3 follows from Statements 1 and
2. Hence it suffices to prove Statement 2. Next we do this. On one
hand, 
\begin{align}
\frac{{n \choose \le r}}{{n \choose r}}=\frac{\sum_{k=0}^{r}{n \choose r-k}}{{n \choose r}} & =\sum_{k=0}^{r}\frac{r...(r-k+1)}{(n-r+k)...(n-r+1)}\nonumber \\
 & =\sum_{k=0}^{r}\frac{\beta...(\beta-\frac{k-1}{n})}{(1-\beta+\frac{k}{n})...(1-\beta+\frac{1}{n})}\leq\sum_{k=0}^{r}(\frac{\beta}{1-\beta})^{k}.\label{eq:-105}
\end{align}
On the other hand, for a fixed $N$ and for any $r/n\in[0,\frac{1}{2}-\delta]$,
\begin{align*}
\frac{{n \choose \le r}}{{n \choose r}}-\sum_{k=0}^{r}(\frac{\beta}{1-\beta})^{k} & \geq\sum_{k=0}^{\min\{r,N\}}\frac{\beta...(\beta-\frac{k-1}{n})}{(1-\beta+\frac{k}{n})...(1-\beta+\frac{1}{n})}-\sum_{k=0}^{r}(\frac{\beta}{1-\beta})^{k}\\
 & \geq\sum_{k=0}^{\min\{r,N\}}[(\frac{\beta-\frac{k}{n}}{1-\beta+\frac{k}{n}})^{k}-(\frac{\beta}{1-\beta})^{k}]-\sum_{k=\min\{r,N\}+1}^{r}(\frac{\beta}{1-\beta})^{k}
\end{align*}
Therefore, for fixed $N$, 
\[
\liminf_{n\to\infty}\inf_{r/n\in[0,\frac{1}{2}-\delta]}\{\frac{{n \choose \le r}}{{n \choose r}}-\sum_{k=0}^{r}(\frac{\beta}{1-\beta})^{k}\}\geq\liminf_{n\to\infty}\inf_{\beta\in[0,\frac{1}{2}-\delta]}-\sum_{k=\min\{r,N\}+1}^{r}(\frac{\beta}{1-\beta})^{k}\geq-\frac{\gamma^{N+1}}{1-\gamma}.
\]
where $\gamma:=\frac{\frac{1}{2}-\delta}{\frac{1}{2}+\delta}<1$.
Since $N$ is arbitrary, 
\begin{equation}
\liminf_{n\to\infty}\inf_{r/n\in[0,\frac{1}{2}-\delta]}\{\frac{{n \choose \le r}}{{n \choose r}}-\sum_{k=0}^{r}(\frac{\beta}{1-\beta})^{k}\}\geq0.\label{eq:-106}
\end{equation}
Combining \eqref{eq:-105} and \eqref{eq:-106} yields that given
$\delta>0$, $\frac{{n \choose \le r}}{{n \choose r}}\to\sum_{k=0}^{r}(\frac{\beta}{1-\beta})^{k}$
uniformly for all $\beta:=r/n\in[0,\frac{1}{2}-\delta]$.

\section*{Acknowledgments}
We thank the anonymous reviewer for his/her careful review  which have helped to greatly improve the paper. 

\bibliographystyle{abbrvnat}
\bibliography{ref}

\begin{thebibliography}{22}
\providecommand{\natexlab}[1]{#1}
\providecommand{\url}[1]{\texttt{#1}}
\expandafter\ifx\csname urlstyle\endcsname\relax
  \providecommand{\doi}[1]{doi: #1}\else
  \providecommand{\doi}{doi: \begingroup \urlstyle{rm}\Url}\fi

\bibitem[Benjamini et~al.(1999)Benjamini, Kalai, and
  Schramm]{benjamini1999noise}
I.~Benjamini, G.~Kalai, and O.~Schramm.
\newblock Noise sensitivity of {Boolean} functions and applications to
  percolation.
\newblock \emph{Publications Math{\'e}matiques de l'Institut des Hautes Etudes
  Scientifiques}, 90\penalty0 (1):\penalty0 5--43, 1999.

\bibitem[Bernstein(1967)]{bernstein1967maximally}
A.~J. Bernstein.
\newblock Maximally connected arrays on the $n$-cube.
\newblock \emph{SIAM Journal on Applied Mathematics}, 15\penalty0 (6):\penalty0
  1485--1489, 1967.

\bibitem[Bezrukov(1999)]{bezrukov1999edge}
S.~L. Bezrukov.
\newblock Edge isoperimetric problems on graphs.
\newblock \emph{Graph Theory and Combinatorial Biology}, 7:\penalty0 157--197,
  1999.

\bibitem[Harper(1964)]{harper1964optimal}
L.~H. Harper.
\newblock Optimal assignments of numbers to vertices.
\newblock \emph{Journal of the Society for Industrial and Applied Mathematics},
  12\penalty0 (1):\penalty0 131--135, 1964.

\bibitem[Harper(1991)]{harper1991problem}
L.~H. Harper.
\newblock On a problem of {Kleitman and West}.
\newblock \emph{Discrete mathematics}, 93\penalty0 (2-3):\penalty0 169--182,
  1991.

\bibitem[Hart(1976)]{hart1976note}
S.~Hart.
\newblock A note on the edges of the $n$-cube.
\newblock \emph{Discrete Mathematics}, 14\penalty0 (2):\penalty0 157--163,
  1976.

\bibitem[Kahn et~al.(1988)Kahn, Kalai, and Linial]{kahn1988influence}
J.~Kahn, G.~Kalai, and N.~Linial.
\newblock The influence of variables on {Boolean} functions.
\newblock In \emph{29th Annual Symposium on Foundations of Computer Science},
  pages 68--80. IEEE, 1988.

\bibitem[Kamath and Anantharam(2016)]{kamath2016non}
S.~Kamath and V.~Anantharam.
\newblock On non-interactive simulation of joint distributions.
\newblock \emph{IEEE Transactions on Information Theory}, 62\penalty0
  (6):\penalty0 3419--3435, 2016.

\bibitem[Kirshner and Samorodnitsky(2021)]{kirshner2021moment}
N.~Kirshner and A.~Samorodnitsky.
\newblock A moment ratio bound for polynomials and some extremal properties of
  {Krawchouk} polynomials and {Hamming} spheres.
\newblock \emph{IEEE Transactions on Information Theory}, 67\penalty0
  (6):\penalty0 3509--3541, 2021.

\bibitem[Kleitman(1966)]{kleitman1966combinatorial}
D.~J. Kleitman.
\newblock On a combinatorial conjecture of {Erd{\"o}s}.
\newblock \emph{Journal of Combinatorial Theory}, 1\penalty0 (2):\penalty0
  209--214, 1966.

\bibitem[Krasikov(2001)]{krasikov2001nonnegative}
I.~Krasikov.
\newblock Nonnegative quadratic forms and bounds on orthogonal polynomials.
\newblock \emph{Journal of Approximation Theory}, 111\penalty0 (1):\penalty0
  31--49, 2001.

\bibitem[Levenshtein(1995)]{levenshtein1995krawtchouk}
V.~I. Levenshtein.
\newblock Krawtchouk polynomials and universal bounds for codes and designs in
  {Hamming} spaces.
\newblock \emph{IEEE Transactions on Information Theory}, 41\penalty0
  (5):\penalty0 1303--1321, 1995.

\bibitem[Lindsey(1964)]{lindsey1964assignment}
J.~H. Lindsey.
\newblock Assignment of numbers to vertices.
\newblock \emph{The American Mathematical Monthly}, 71\penalty0 (5):\penalty0
  508--516, 1964.

\bibitem[MacWilliams and Sloane(1977)]{macwilliams1977theory}
F.~J. MacWilliams and N.~J.~A. Sloane.
\newblock \emph{The Theory of Error-Correcting Codes}, volume~16.
\newblock Elsevier, 1977.

\bibitem[O'Donnell(2014)]{O'Donnell14analysisof}
R.~O'Donnell.
\newblock \emph{Analysis of {Boolean} Functions}.
\newblock Cambridge University Press, 2014.

\bibitem[Polyanskiy(2019)]{polyanskiy2019hypercontractivity}
Y.~Polyanskiy.
\newblock Hypercontractivity of spherical averages in {Hamming} space.
\newblock \emph{SIAM Journal on Discrete Mathematics}, 33\penalty0
  (2):\penalty0 731--754, 2019.

\bibitem[Rashtchian and Raynaud(2019)]{rashtchian2019edge}
C.~Rashtchian and W.~Raynaud.
\newblock Edge isoperimetric inequalities for powers of the hypercube.
\newblock \emph{arXiv preprint arXiv:1909.10435}, 2019.

\bibitem[Witsenhausen(1975)]{witsenhausen1975sequences}
H.~S. Witsenhausen.
\newblock On sequences of pairs of dependent random variables.
\newblock \emph{SIAM Journal on Applied Mathematics}, 28\penalty0 (1):\penalty0
  100--113, 1975.

\bibitem[Yu(2021)]{yu2021strong}
L.~Yu.
\newblock Strong {Brascamp-Lieb} inequalities.
\newblock \emph{arXiv preprint arXiv:2102.06935}, 2021.

\bibitem[Yu and Tan(2019)]{yu2019improved}
L.~Yu and V.~Y.~F. Tan.
\newblock An improved linear programming bound on the average distance of a
  binary code.
\newblock \emph{arXiv preprint arXiv:1910.09416}, 2019.

\bibitem[Yu and Tan(2021)]{yu2021non}
L.~Yu and V.~Y.~F. Tan.
\newblock On non-interactive simulation of binary random variables.
\newblock \emph{IEEE Transactions on Information Theory}, 67\penalty0
  (4):\penalty0 2528--2538, 2021.

\bibitem[Yu et~al.(2021)Yu, Anantharam, and Chen]{yu2021graphs}
L.~Yu, V.~Anantharam, and J.~Chen.
\newblock Graphs of joint types, noninteractive simulation, and stronger
  hypercontractivity.
\newblock \emph{arXiv preprint arXiv:2102.00668}, 2021.

\end{thebibliography}

\end{document}